\documentclass[12pt]{amsart}
\usepackage{amssymb}
\usepackage{amsmath,amscd}

\let\frak\mathfrak
\let\Bbb\mathbb

\def\>{\relax\ifmmode\mskip.666667\thinmuskip\relax\else\kern.111111em\fi}
\def\<{\relax\ifmmode\mskip-.333333\thinmuskip\relax\else\kern-.0555556em\fi}
\def\vsk#1>{\vskip#1\baselineskip}
\def\vv#1>{\vadjust{\vsk#1>}\ignorespaces}
\def\vvn#1>{\vadjust{\nobreak\vsk#1>\nobreak}\ignorespaces}
\def\vvgood{\vadjust{\penalty-500}}
\let\alb\allowbreak
\def\qedph{\hphantom{\quad\qedsymbol}}

\def\sskip{\par\vskip.2\baselineskip plus .05\baselineskip}
\let\Medskip\medskip
\def\medskip{\par\Medskip}
\let\Bigskip\bigskip
\def\bigskip{\par\Bigskip}

\let\Maketitle\maketitle
\def\maketitle{\hrule height0pt\vskip-\baselineskip
\Maketitle\thispagestyle{empty}\let\maketitle\empty}

\newtheorem{thm}{Theorem}[section]
\newtheorem{cor}[thm]{Corollary}
\newtheorem{lem}[thm]{Lemma}
\newtheorem{prop}[thm]{Proposition}

\numberwithin{equation}{section}

\theoremstyle{definition}
\newtheorem*{rem}{Remark}

\let\mc\mathcal
\let\nc\newcommand

\nc{\on}{\operatorname}
\nc{\Z}{{\mathbb Z}}
\nc{\C}{{\mathbb C}}
\nc{\N}{{\mathbb N}}
\nc{\pone}{{\mathbb C}{\mathbb P}^1}
\nc{\arr}{\rightarrow}
\nc{\larr}{\longrightarrow}
\nc{\al}{\alpha}
\nc{\W}{{\mc W}}
\nc{\la}{\lambda}
\nc{\su}{\widehat{{\mathfrak sl}}_2}
\nc{\g}{{\mathfrak g}}
\nc{\h}{{\mathfrak h}}
\nc{\m}{{\mathfrak m}}
\nc{\n}{{\mathfrak n}}
\nc{\Gm}{\Gamma}
\nc{\La}{\Lambda}
\nc{\gl}{\widehat{\mathfrak{gl}_2}}
\nc{\bi}{\bibitem}
\nc{\om}{\omega}
\nc{\Res}{\on{Res}}
\nc{\gm}{\gamma}
\nc{\Om}{\Omega}

\def\z{\mathfrak z}
\def\ev{\mbox{\sl ev}}
\def\ch{\on{ch}}

\def\End{\on{End\>}}

\def\Res{\on{Res}}

\def\Wr{\on{Wr}}

\def\tbigcap{\mathrel{\textstyle{\bigcap}}}

\def\lba{{\bs\la\>,\bs a}}

\def\B{{\mc B}}
\def\D{{\mc D}}

\def\F{{\mc F}}

\def\M{{\mc M}}
\def\O{{\mc O}}

\def\V{{\mc V}}

\let\dl\delta
\let\Dl\Delta

\let\si\sigma

\let\der\partial
\let\minus\setminus

\let\ge\geqslant
\let\geq\geqslant

\let\leq\leqslant

\nc{\gln}{\mathfrak{gl}_N}
\nc{\sln}{\mathfrak{sl}_N}

\def\gltt{\glt[t]}

\def\beq{\begin{equation}}
\def\eeq{\end{equation}}
\def\be{\begin{equation*}}
\def\ee{\end{equation*}}

\nc{\bean}{\begin{eqnarray}}
\nc{\eean}{\end{eqnarray}}
\nc{\bea}{\begin{eqnarray*}}
\nc{\eea}{\end{eqnarray*}}
\nc{\bs}{\boldsymbol}
\nc{\Ref}[1]{{\rm(\ref{#1})}}
\nc{\glN}{\mathfrak{gl}_N}
\nc{\glNt}{\mathfrak{gl}_N[t]}
\nc{\s}{sing}
\nc{\R}{\Bbb R}

\nc{\Oml}{{\Om_{\bs\la}}}
\nc{\OmLb}{{\Om_{\bs\La,\bs\la,\bs b}}}
\nc{\Ol}{{\mc O_{\bs\la}}}
\nc{\OLb}{{\mc O_{\bs\La,\bs\la,\bs b}}}
\nc{\VSl}{{(\V^S)_{\bs\la}}}
\nc{\Bl}{{\B_{\bs\la}}}
\nc{\Ml}{{\mc M_{\bs\la}}}
\nc{\Mlb}{{\mc M_{\bs\La,\bs\la,\bs b}}}
\nc{\Blb}{{\B_{\bs\La,\bs\la,\bs b}}}
\nc{\Omn}{{\Omega_{\bs n,\bs b,\bs K}}}
\nc{\Omlb}{{\bar\Om_{\bs\la}}}
\nc{\ep}{\epsilon}

\nc{\Dlb}{\Dl_{\bs\La,\bs\la,\bs b,\bs K}}
\nc{\Bb}{{\bf b}}
\nc{\glt}{{\frak{gl}_2}}
\nc{\A}{{\mc A}}
\nc{\slt}{{\frak{sl}_2}}
\nc{\Ma}{{\mc M_{\bs a}}}
\nc{\Mal}{{\mc M_{\bs\la,\bs a}}}
\nc{\Malp}{{\mc M_{\phi,\bs\la,\bs a}}}
\nc{\Vl}{{\V^S_{\bs\la}}}
\nc{\Bal}{{\B_{\bs\la,\bs a}}}
\nc{\Ola}{{\mc O_{\bs\la,\bs a}}}
\nc{\Bv}{{\mc B_{\V^S}}}
\nc{\Bvz}{{\mc B^0_{\V^S}}}

\nc{\sing}{{\rm Sing\,}}
\nc{\Uglt}{U(\glt)}
\nc{\Olo}{{\mc O^0_{\bs\la}}}

\begin{document}

\title[The 
Bethe algebra associated with a nilpotent element]
{The $\frak{gl}_2$
Bethe algebra associated with a nilpotent element}

\author[E.\,Mukhin, V.\,Tarasov,  A.\,Varchenko]
{E.\,Mukhin$\>^{*,1}$, V.\,Tarasov$\>^{\star,*}$,
and A.\,Varchenko$\>^{\diamond,2}$}

\thanks{${}^1$\ Supported in part by NSF grant DMS-0601005}

\thanks{${}^2$\ Supported in part by NSF grant DMS-0555327}

\maketitle

\begin{center}
{\it $^\star\<$Department of Mathematical Sciences
Indiana University\,--\>Purdue University Indianapolis\\
402 North Blackford St, Indianapolis, IN 46202-3216, USA\/}

\medskip
{\it $^*\<$St.\,Petersburg Branch of Steklov Mathematical Institute\\
Fontanka 27, St.\,Petersburg, 191023, Russia\/}

\medskip
{\it $^\diamond\<$Department of Mathematics, University of North Carolina
at Chapel Hill\\ Chapel Hill, NC 27599-3250, USA\/}
\end{center}

\medskip
\begin{abstract}
To any $2\times 2$-matrix $\bs K$
one assigns a commutative subalgebra
$\B^{\bs K}\subset U(\frak{gl}_2[t])$ called a Bethe algebra.
We describe relations between the Bethe algebras,
associated with the zero matrix and a nilpotent
matrix.
\end{abstract}

\maketitle

\section{Introduction} To any $N\times N$-matrix $\bs K$
one assigns a commutative subalgebra
$\B^{\bs K}\subset U(\frak{gl}_N[t])$ called a Bethe algebra \cite{T},
\cite{MTV1}, \cite{CT}. The Bethe
algebra acts on any $U(\frak{gl}_N[t])$-module giving an example of a
quantum integrable system. In particular, it acts on any evaluation
$U(\frak{gl}_N[t])$-module $L_{\bs\la}(0)$, where $L_{\bs\la}$ is the
irreducible finite-dimensional $\gln$-module with some highest dominant integral
weight $\bs\la$.

The most interesting of the Bethe algebras is the Bethe algebra $\B^0$
associated with the zero matrix $\bs K$. The Bethe algebra $\B^0$ is
closely connected with Schubert calculus in Grassmannians of
$N$-dimensional subspaces. The eigenvectors of the $\B^0$-action on
suitable
$U(\frak{gl}_N[t])$-modules are in a bijective correspondence
with intersection points of suitable Schubert cycles \cite{MTV3},
\cite{MTV4}. The most important of those $U(\frak{gl}_N[t])$-modules
is the infinite-dimensional module $\V^S = (V^{\otimes n}\otimes
\C[z_1,\dots,z_n])^S$ introduced in \cite{MTV3}. Here $V^{\otimes n}$
is the $n$-fold tensor power of the vector representation of $\glN$
and the upper index $S$ denotes the subspace of invariants with
respect to a natural action of the symmetric group $S_n$. The other
$U(\frak{gl}_N[t])$-modules related to Schubert calculus are
subquotients of $\V^S$.

The Bethe algebra $\B^0$ commutes with the
subalgebra $U(\frak{gl}_N) \subset U(\frak{gl}_N[t])$. Let
$\V^S = \oplus _{\bs\la} \V^{S,0}_{\bs\la}$
be the $\gln$-isotypical
decomposition, where $\bs\la$ runs through $\gln$-highest weights. The Bethe
algebra $\B^0$ preserves this decomposition and
$\B^0_{\V^S} = \oplus_{\bs\la}\B_{\bs\la}^0$,
where $\B^0_{\V^S}\subset \End (\V^S)$ and
$\B^0_{\bs\la}\subset \End (\V^{S,0}_{\bs\la})$ are the images of $\B^0$.
It is shown in \cite{MTV3} that the Bethe
algebra $\B^0_{\bs\la}$ is isomorphic to the algebra $\O^0_{\bs\la}$
of
functions on a suitable Schubert cell $\Omega_{\bs\la}$ in a Grassmannian.
It is also shown that the
$\B^0_{\bs\la}$-module $\V^{S,0}_{\bs\la}$ is isomorphic to the
regular representation of $\O^0_{\bs\la}$. These statements
give a geometric
interpretation of the $\B^0_{\bs\la}$-module $\V^{S,0}_{\bs\la}$
(or representational interpretation of $\O^0_{\bs\la}$)
and they are key facts for applications of Bethe algebras to Schubert calculus.

This paper has two goals.

The first is to extend these results to the Bethe algebras
$\B^{\bs K}$ associated with nonzero matrices $\bs K$.
Note that this goal was accomplished in \cite{MTV5}
for diagonal matrices $\bs K$ with distinct diagonal entries.

The second goal
is to express
the $\B^{\bs K}$-action on the infinite-dimensional module
$\V^S$ in terms of the
$\B^{0}$-action on $\V^S$ and the $\B^{\bs K}$-actions on
finite-dimensional modules $L_{\bs\la}(0)$.

In this paper we achieve these two goals for one example: $N=2$ and $
\bs K =
\left( \begin{matrix}
{} 0 & 0\\
-1 & 0
\end{matrix}\right)$.

We denote $\B$ the Bethe algebra $\B^{\bs K}$
associated with that nilpotent matrix $\bs K$.
We define a decomposition $\V^S = \oplus_{\bs\la} \V^S_{\bs\la}$ into
suitable $\B$-modules called
the deformed isotypical components of $\V^S$. For any $\bs\la$,
$\V^{S}_{\bs\la}$ is a suitable deformation
of the isotypical component $\V^{S,0}_{\bs\la}$. In particular,
$\V^{S}_{\bs\la}$ and $\V^{S,0}_{\bs\la}$ have equal ranks
as $\C[z_1,\dots,z_n]^S$-modules.
We have $\B_{\V^S} = \oplus_{\bs\la} \B_{\bs\la}$,
where
$\B_{\V^S}\subset \End (\V^S)$ and
$\B_{\bs\la}\subset \End (\V^{S}_{\bs\la})$ are the images of $\B$.

For any $\bs\la = (n-k,k)$, the image of $\B$ in
$L_{\bs\la}(0)$ is isomorphic to $\A_{n-2k} = \C[b]/\langle b^{n-2k+1}\rangle$.
The algebra $\A_{n-2k}$ acts on $L_{\bs\la}(0)$ by the formula
$b\mapsto e_{21}$ where $e_{21}$ is one of the four standard generators of $\glt$.
We show that the $\B_{\bs\la}$-module $\V^S_{\bs\la}$ is isomorphic to the regular
representation of $\A_{n-2k}\otimes \O^0_{\bs\la}$,
Theorems \ref{first} and \ref{first1}.
This statement gives a
geometric interpretation of
the $\B_{\bs\la}$-module $\V_{\bs\la}$ as the regular
representation of the algebra of functions with nilpotents on the Schubert cell
$\Omega_{\bs\la}$, where the nilpotents are determined by the algebra
$\A_{n-2k}$.
This statement is our achievement of the first goal of this paper.

We define an action
of $\A_{n-2k}\otimes \B^0_{\bs\la}$ on $\V^{S,0}_{\bs\la}$ by the formula
$b^j\otimes B : v \mapsto (e_{21})^jBv$.
The $\A_{n-2k}\otimes \B^0_{\bs\la}$-module $\V^{S,0}_{\bs\la}$
is isomorphic to the regular
representation of $\A_{n-2k} \otimes \O^0_{\bs\la}$ due to Theorems 5.3 and 5.6
in \cite{MTV3}.

As a result of these descriptions
of the
$\B_{\bs\la}$-module $\V^S_{\bs\la}$ and
$\A_{n-2k}\otimes \B^0_{\bs\la}$-module $\V^{S,0}_{\bs\la}$,
we construct an algebra isomorphism
$\nu_{\bs\la} : \A_{n-2k}\otimes B_{\bs\la}^0 \to B_{\bs\la}$
and a linear isomorphism $\eta_{\bs\la} : \V^{S,0}_{\bs\la}
\to \V^{S}_{\bs\la}$ which establish an isomorphism of the
$\B_{\bs\la}$-module $\V^S_{\bs\la}$ and
$\A_{n-2k}\otimes \B^0_{\bs\la}$-module $\V^{S,0}_{\bs\la}$,
see Theorem \ref{last thm}. This statement is our achievement of
the second goal.

\medskip

The paper is organized as follows. In Section \ref{alg sec}, we
discuss representations of $U(\gltt)$ and introduce
the $U(\gltt)$-module $\V^S$. We introduce the Bethe algebra $\B^{\bs K}$
in Section \ref{Bethe sec}. We define decompositions $\V^S =
\oplus_{\bs\la} \V^S_{\bs\la}$ and $\B_{\V^S} = \oplus_{\bs\la}
\B_{\bs\la}$ in Section \ref{sec actionss}. We study
deformed isotypical components in Section \ref{sec More}.
Section \ref{sec Ol} is on the algebra
$\Ol \simeq \A_{n-2k}\otimes \Olo$.
The first connections between the algebras $\B_{\bs\la}$
and $\O_{\bs\la}$ are discussed in Section \ref{Ol and B}.
In Section \ref{sec ISOMORR} we show that the
the $\B_{\bs\la}$-module $\V^S_{\bs\la}$ is isomorphic to the regular
representation of $\A_{n-2k}\otimes \O^0_{\bs\la}$.
In Section \ref{sec comparison} we show that the
$\B_{\bs\la}$-module $\V^S_{\bs\la}$ and
$\A_{n-2k}\otimes \B^0_{\bs\la}$-module $\V^{S,0}_{\bs\la}$ are isomorphic.

\medskip
In \cite{FFR}, the authors study the Bethe algebra associated with
a principal nilpotent element. One of
our motivations was
to relate the picture in \cite {FFR} with our description of Bethe
algebras in \cite{MTV3}, \cite{MTV5}.

\section{Representations of current algebra $\gltt$}
\label{alg sec}
\subsection{Lie algebra $\glt$}
Let $e_{ij}$, $i,j=1,2$, be the standard generators of the complex
Lie algebra
$\glt$ satisfying the relations
$[e_{ij},e_{sk}]=\dl_{js}e_{ik}-\dl_{ik}e_{sj}$. We identify the Lie algebra
$\slt$ with the subalgebra in $\glt$ generated by the elements
$e_{11}-e_{22}, e_{12}, e_{21}$.

The elements
$e_{11}+e_{22}$ \,and $(e_{11}+1)e_{22}-e_{21}e_{12}$
are free generators of the center of\/ $\Uglt$.

Let $M$ be a $\glt$-module. A vector $v\in M$ has weight
$\bs\la=(\la_1,\la_2)\in\C^2$ if $e_{ii}v=\la_iv$ for $i=1,2$.
A vector $v$ is called {\it singular\/} if $e_{12}v=0$.

We denote by $M[\bs\la]$ the subspace of $M$ of weight $\bs\la$,
by $\sing M$ the subspace of $M$ of all singular vectors and by
$\sing M[\bs\la]$ the subspace of $M$ of all singular vectors
of weight $\bs\la$.

Denote $L_{\bs\la}$ the irreducible finite-dimensional $\glt$-module with
highest weight $\bs\la$. Any finite-dimensional $\glt$
weight module $M$ is isomorphic
to the direct sum $\bigoplus_{\bs\la}L_{\bs\la}\otimes\sing M[\bs\la]$,
where the spaces $\sing M[\bs\la]$ are considered as trivial
$\glt$-modules.

The $\glt$-module $L_{(1,0)}$ is the standard $2$-dimensional vector
representation of $\glt$. We denote it $V$. We choose a highest weight
vector of $V$ and denote it $v_+$. A $\glt$-module $M$ is called
polynomial if it is isomorphic to a submodule of $V^{\otimes n}$ for
some $n$.

A sequence of integers $\bs\la=(\la_1,\la_2)$,
$\la_1\ge\la_2\ge0$,\ is called a {\it partition with at most
$2$ parts\/}. Denote $|\bs\la|=\la_1 +\la_2$. We say that $\bs\la$
is a partition of $|\bs\la|$.

\sskip
The $\glt$-module $V^{\otimes n}$ contains the module $L_{\bs\la}$
if and only if $\bs\la$ is a partition of $n$ with at most $2$ parts.

\sskip
For a Lie algebra $\g\,$, we denote $U(\g)$ the universal enveloping algebra
of $\g$.

\subsection{Current algebra $\gltt$}
\label{current}
Let $\gltt=\glt\otimes\C[t]$ be the complex
Lie algebra of $\glt$-valued polynomials
with the pointwise commutator.
We identify $\glt$ with the subalgebra $\glt\otimes1$
of constant polynomials in $\gltt$. Hence, any $\gltt$-module has a canonical
structure of a $\glt$-module.

The standard generators of $\gltt$ are $e_{ij}\otimes t^r$, $i,j=1,2$,
$r\in\Z_{\ge0}$. They satisfy the relations
$[e_{ij}\otimes t^r,e_{sk}\otimes t^p]=
\dl_{js}e_{ik}\otimes t^{r+p}-\dl_{ik}e_{sj}\otimes t^{r+p}$.

The subalgebra $\z_2[t]\subset\gltt$ generated by the elements
$(e_{11}+e_{22})\otimes t^r$, $r\in\Z_{\ge0}$, is central.
The Lie algebra $\gltt$ is canonically isomorphic to the direct sum
$\slt[t]\oplus\z_2[t]$.

It is convenient to collect elements of $\gltt$ in generating series
of a variable $u$. For $g\in\glt$, set
\vvn-.2>
\be
g(u)=\sum_{s=0}^\infty (g\otimes t^s)u^{-s-1}.
\vv.2>
\ee
We have $(u-v)[e_{ij}(u),e_{sk}(v)] =
\dl_{js}(e_{ik}(u)-e_{ik}(v)) -
\dl_{ik}(e_{sj}(u)-e_{sj}(v))$.

For each $a\in\C$, there is an automorphism $\rho_a$ of $\gltt$,
\;$\rho_a:g(u)\mapsto g(u-a)$. Given a $\gltt$-module $M$, we denote by $M(a)$
the pull-back of $M$ through the automorphism $\rho_a$. As $\glt$-modules,
$M$ and $M(a)$ are isomorphic by the identity map.

For any $\gltt$-modules $L,M$ and any $a\in\C$, the identity map
$(L\otimes M)(a)\to L(a)\otimes M(a)$ is an isomorphism of $\gltt$-modules.

We have the evaluation homomorphism,
${\ev:\gltt\to\glt}$, \;${\ev:g(u) \mapsto g\>u^{-1}}$.
Its restriction to the subalgebra $\glt\subset\gltt$ is the identity map.
For any $\glt$-module $M$, we denote by the same letter the $\gltt$-module,
obtained by pulling $M$ back through the evaluation homomorphism. Then for each
$a\in\C$, the $\gltt$-module $M(a)$ is called an {\it evaluation module\/}.

\medskip

Define a grading on $\gltt$ such that the degree of
$e_{ij}\otimes t^r$ equals $r+j-i$ for all $i,j,r$.
We set the degree of $u$ to be $1$.
Then the series $g(u)$ is homogeneous of degree $j-i-1$.

A $\gltt$-module is called {\it graded\/} if it
has a bounded from below $\Z$-grading compatible with the grading on $\gltt$.
Any irreducible graded $\gltt$-module is isomorphic to an evaluation module
$L(0)$ for some irreducible $\glt$-module $L$, see~\cite{CG}.

Let $M$ be a $\Z$-graded space with finite-dimensional homogeneous
components. Let $M_j\subset M$ be the homogeneous component of degree $j$.
We call the Laurent series in a variable $q$,
\vvn-.2>
\be
\ch_M(q)\ =\ \sum_{j}\ (\dim M_j)\,q^j\,,
\vv-.2>
\ee
the {\it graded character\/} of $M$.

\subsection{Weyl modules}
\label{secweyl}
Let $W_m$ be the $\gltt$-module generated by a vector $v_m$ with the
defining relations:
\vvn-.2>
\bea
e_{11} (u)v_m= \frac mu\,v_m\,, \qquad
e_{22} (u)v_m=\>0\,,
\\
e_{12} (u)v_m=\>0\,,\qquad (e_{21} \otimes1)^{m+1}v_m=\>0\,.
\eea
As an $\slt[t]$-module, the module $W_m$ is isomorphic to the Weyl module from
\cite{CL}, \cite{CP}, corresponding to the weight $m\om$, where $\om$ is
the fundamental weight of $\slt$. Note that $W_1=V(0)$.

\begin{lem}[\cite{CP}, cf. \cite{MTV3}]
\label{weyl}
The module $W_m$ has the following properties.
\begin{enumerate}
\item[(i)] The module \>$W_m$ has a unique grading such that\/ \>$W_m$ is
a graded $\gltt$-module and the degree of\/ $v_m$ equals\/ $0$.

\item[(ii)] As a $\glt$-module, $W_m$ is isomorphic to $V^{\otimes m}$.

\item[(iii)] A $\gltt$-module $M$ is an irreducible subquotient of\/ $W_m$
if and only if\/ $M$ has the form $L_{\bs\la}(0)$,
where\/ $\bs\la$ is a partition of\/ $m$ with at most \/ $2$ parts.

\item[(iv)]
Consider the decomposition of $W_m$ into isotypical components of the
$\glt$-action, \\
$W_m = \oplus_{\bs\la} (W_m)_{\bs\la}$,
where $(W_m)_{\bs\la}$
is the isotypical component corresponding
to the irreducible polynomial $\glt$-module
with highest weight $\bs\la = (m-k,k)$.
Then for any $\bs\la$,
the graded character of $(W_m)_{\bs\la}$ is given by
\vvn-.2>
\be
\ch_{(W_m)_{\bs\la}}(q)\,=\,
\frac{(1-q^{m-2k+1})^2}
{1-q}\,
\frac{(q)_m }
{(q)_{m-k+1}(q)_{k}}\,q^{2k-m}\ ,
\ee
where\/ $\,(q)_a=\prod_{j=1}^a(1-q^j)\,$.
\end{enumerate}
\end{lem}
\begin{proof}
A proof follows from Lemma 2.2 in \cite{MTV3}.
\vvgood
\end{proof}

Given sequences $\bs n=(n_1,\dots,n_k)$ of natural numbers and
$\bs b=(b_1,\dots,b_k)$ of distinct complex numbers, we call the $\gltt$-module
$\otimes_{s=1}^k W_{n_s}(b_s)$ the {\it Weyl module associated with\/ $\bs n$
and\/ $\bs b$}.

\subsection{$\gltt$-module $\V^S$}
\label{VS}

Let $\V$ be the space of polynomials in $z_1,\dots,z_n$ with coefficients
in $V^{\otimes n}$,
\vvn-.3>
\be
\V\>=\,V^{\otimes n}\<\otimes_{\C}\C[z_1,\dots,z_n]\,.
\vv.2>
\ee
The space $V^{\otimes n}$ is embedded in $\V$ as the subspace of constant
polynomials.

For $v\in V^{\otimes n}$ and
$p(z_1,\dots,z_n)\in\C[z_1,\dots,z_n]$, we write
$p(z_1,\dots,z_n)\,v$ to denote $v\otimes p(z_1,\dots,z_n)$.

The symmetric group $S_n$ acts on $\V$ by permuting the factors
of $V^{\otimes n}$ and the variables $z_1,\dots,z_n$ simultaneously,
\vvn.2>
\be
\si\bigl(p(z_1,\dots,z_n)\,v_1\otimes\dots\otimes v_n\bigr)\,=\,
p(z_{\si(1)},\dots,z_{\si(n)})\,
v_{\si^{-1}(1)}\!\otimes\dots\otimes v_{\sigma^{-1}(n)}\,,\qquad\si\in S_n\,.
\kern-3em
\vv.2>
\ee
We denote $\V^S$ the subspace of $S_n$-invariants of $\V$.

\begin{lem}[\cite{CP}, cf. \cite{MTV3}]
\label{VSfree}
The space $\V^S$ is a free $\C[z_1,\dots,z_n]^S$-module of rank $2^n$.
\end{lem}

We consider the space $\V$ as a $\gltt$-module with
a series $g(u)$, \,$g\in\glt$, acting by
\vvn.1>
\beq
\label{action}
g(u)\,\bigl(p(z_1,\dots,z_n)\,v_1\otimes\dots\otimes v_n)\,=\,
p(z_1,\dots,z_n)\,\sum_{s=1}^n
\frac{v_1\otimes\dots\otimes gv_s\otimes\dots\otimes v_n}{u-z_s}\ .
\vv.2>
\eeq
The $\gltt$-action on $\V$ commutes with the $S_n$-action.
Hence, $\V^S$ is a $\gltt$-submodule of $\V$.

\medskip
The space $\mc V^S$ as an $sl_2$-module was introduced and studied in \cite{CP}.

\subsection{Weyl modules as quotients of $\V^S$}
\label{Weyl modules as quotients}
Let $\si_s(\bs z)$, $s=1,\alb\dots,n$, be the $s$-th elementary
symmetric polynomial in $z_1,\dots,z_n$.
For $\bs a=(a_1,\dots,a_n)\in\C^n$, denote
$I_{\bs a}\subset\C[z_1,\dots,z_n]$ the ideal generated by
the polynomials $\si_s(\bs z)-a_s$, $s=1,\dots, n$.
Define
\vvn.3>
\beq
\label{IVa}
I^\V_{\bs a}=(V^{\otimes n}\otimes I_{\bs a})\tbigcap \V^S. \vv.2>
\eeq
Clearly, ${I^\V_{\bs a}}$ is a $\gltt$-submodule of $\V^S$ and a
free $\C[z_1,\dots,z_n]^S$-module.

Define distinct complex numbers $b_1,\dots,b_k$ and
natural numbers $n_1,\dots,n_k$ by the relation
\vvn-.6>
\beq\label{ab}
\prod_{s=1}^k\,(u-b_s)^{n_s}=\,u^n+\>\sum_{j=1}^n (-1)^j\>a_j\>u^{n-j}.
\vv-.7>
\eeq
Clearly, $\sum_{s=1}^kn_s=n$.

\begin{lem}[\cite{CP}, cf. \cite{MTV3}]
\label{factor=weyl}
The $\gltt$-modules $\V^S/{I^\V_{\bs a}}$ and
$\otimes_{s=1}^kW_{n_s}(b_s)$ are isomorphic.
\end{lem}

\subsection{Grading on $\V^S$}
\label{sec grading on V}

Let $V^{\otimes n} =\oplus_{k=0}^n V^{\otimes n}[n-k,k]$ be the
$\glt$-weight decomposition. Define a grading on $V^{\otimes n}$ by
setting $\deg v = -k$ for any $v\in V^{\otimes n}[n-k,k]$.
Define a
grading on $\C[z_1,\dots,z_n]$ by setting $\deg z_i=1$ for all
$i=1,\dots,n$.
Define a grading on $\V$ by setting $\deg(v\otimes
p)=\deg v +\deg p$ for any $v\in V^{\otimes n}$ and
$p\in\C[z_1,\dots,z_n]$. The grading on $\V$ induces a grading
on $\V^S$ and $\End(\V^S)$.

\begin{lem}[\cite{CP}]
\label{cycl grad}
The $\gltt$-action on \/ $\V^S$ is graded.
\qed
\end{lem}

\section{Bethe algebra}
\label{Bethe sec}

\subsection{Definition}
\label{bethesec}
Let $\bs K=(K_{ij})$ be a $2\times 2$-matrix with complex coefficients.
Consider the series
\be
B_i^{\bs K}(u)\,=\,\sum_{j=0}^\infty B_{ij}^{\bs K}\>u^{-j}\,,\qquad
i=1,2\,,
\vv.2>
\ee
where $B_{ij}^{\bs K}\in U(\gltt)$, defined by the formulae
\vvn.3>
\begin{gather*}
B^{\bs K}_1(u)\,=\,K_{11}+K_{22}-e_{11}(u)-e_{22}(u) \;,
\\[4pt]
B^{\bs K}_2(u)\,=\,\bigl(K_{11}+e_{11}(u)\bigr)\>\bigl(K_{22}+e_{22}(u)\bigr)
-\bigl(K_{12}+e_{21}(u)\bigr)\>\bigl(K_{21}+e_{12}(u)\bigr) - e_{22}'(u)\;,
\\[-10pt]
\end{gather*}
where $\;'$ stands for the derivative $d/du$. We call the unital subalgebra of
$U(\gltt)$ generated by $B_{ij}^{\bs K} $, $i=1,2$, $j\in\Z_{\geq 0}$,
the {\it Bethe algebra\/} associated with the matrix $\bs K$ and denote it
$\B^{\bs K}$. The elements $B_{ij}^{\bs K}$ will be called the {\it standard
generators} of $\B^{\bs K}$.
\begin{thm}
\label{T-thm}
For any matrix $\bs K$, the algebra $\B^{\bs K}$ is commutative.
If $\bs K$ is the zero matrix, then $\B^{\bs K}$ commutes with
the subalgebra $U(\glt)\subset U(\gltt)$.
\end{thm}
\begin{proof}
Straightforward.
\end{proof}

Let $\der$ be the operator of differentiation with respect to a variable $u$.
An important object associated with the Bethe algebra is the {\it universal
differential operator\/}
\be
\D^{\bs K}=\,\der^2+ B_1^{\bs K}(u)\der + B_2^{\bs K}(u)\;,
\vv-.5>
\ee
see \cite{T}, \cite{CT}, \cite{MTV1}. It is a differential operator with
respect to the variable $u$.

\medskip
If $M$ is a $\B^{\bs K}$-module, we call the image of
$\B^{\bs K}$ in $\End (M)$ the {\it Bethe algebra} of $M$.
The {\it
universal differential operator }of a $\B^{\bs K}$-module $M$ is the
differential operator
\bea
\D\ =\ \der^2+ \bar B_1(u)\der + \bar B_2(u)\ ,
\qquad
\bar B_i(u)\,=\,\sum_{j=0}^\infty \,(B_{ij}^{\bs K})|_M\,u^{-j}\ .
\eea

It is an interesting problem to describe the algebra $\B^{\bs K}$.
In this paper we will consider the cases
\vvn-.3>
\beq
\label{ex of matrices}
\bs K =
\left( \begin{matrix}
0 & 0\\
0 & 0
\end{matrix}\right)
\qquad
{\rm and }
\qquad
\bs K =
\left( \begin{matrix}
{} 0 & 0\\
-1 & 0
\end{matrix}\right)
\vv.2>
\eeq
and will compare the corresponding objects
$\D^{\bs K}, \B^{\bs K}, B_{ij}^{\bs K}$, etc.
The objects associated with the zero
matrix $\bs K$ will be denoted $\D^{0}, \B^{0}, B_{ij}^{0}$, etc.,
while the objects associated with the nonzero matrix $\bs K$
in \Ref{ex of matrices} will be denoted $\D, \B, B_{ij}$, etc.

We have
\vvn-.5>
\begin{align*}
B_1^0(u)\,&{}=\,B_1(u)\,=\,- e_{11}(u) - e_{22}(u) \;,
\\[4pt]
B_2^0(u)\,&{}=\,
e_{11}(u)e_{22}(u) - e_{21}(u)e_{12}(u) - e_{22}'(u) \;,
\\[4pt]
B_2(u)\,&{}=\,B_2^0(u)+e_{21}(u)
\\[-18pt]
\end{align*}
Writing \,$B_i^0(u) \,=\>\sum_{j} B_{ij}^0u^{-j}$ \,and
\,$B_i(u),=\>\sum_{j} B_{ij}u^{-j}$\>, we have
\vvn.2>
\beq
\label{nilp formula}
B_{1,j}\,=\,B^0_{1,j}\;, \qquad
B_{2,j}\,=\,B^0_{2,j} + e_{21}\otimes t^{j-1}\;,
\vv.2>
\eeq
for all $j$. Note that the elements
\vvn.4>
\beq
\label{Bii}
B_{11}^0=-e_{11}-e_{22}\qquad\text{and}\qquad
B_{22}^0=(e_{11}+1)e_{22}-e_{21}e_{12}
\vv.3>
\eeq
belong to the center of the subalgbra $U(\glt)$.

\subsection{Actions of $\B$ and $\B^0$ on $L_{\bs\la}(b)$}
\label{sec action on one mod}
For $b\in\C$ and $\bs\la = (n-k,k)$, consider the action of
the Bethe algebras
$\B$ and $\B^0$ on the evaluation module $L_{\bs\la}(b)$.

\begin{lem}
\label{lem on one module}
${}$

\begin{enumerate}
\item[(i)]
The image of $\B^0$ in $\End(L_{\bs\la})$ is
the subalgebra of scalar operators.
\item[(ii)]
The image of $\B$ in $\End(L_{\bs\la})$ is
the unital subalgebra generated by the element $e_{21}|_{L_{\bs\la}}$.
\end{enumerate}
\end{lem}
\begin{proof}
Part (i) follows from Schur's lemma and the fact that $\B^0$ commutes with
$U(\glt)$.
Part (ii) follows from commutativity of $\B$
and the fact that the image of $B_{21}$ in
$\End(L_{\bs\la})$ equals the image of $e_{21}$.
\end{proof}

\begin{cor}
\label{cor on A_d}
The map $B_{21}|_{L_{\bs\la}} \mapsto b$ defines an isomorphism of the image of
$\B$ in $\End(L_{\bs\la})$ and the algebra $\C[b]/\langle b^{n-2k+1}\rangle$
\qed
\end{cor}

\section {Actions of $\B^0$ and $\B$ on $\V^S$}
\label{sec actionss}
\subsection{Gradings on $\B$ and $\B^0$}
In Section \ref{current}, we introduced a grading on $\gltt$ such that
$\deg\,e_{ij}\otimes t^r = r+j-i$ for all $i,j,r$.

\begin{lem}
\label{lem Bethe degree}
For any $(i,j)$, the elements $B_{ij}^0, B_{ij}\in U(\gltt)$ are
homogeneous of degree $j-i$. \qed
\end{lem}

By Lemma \ref{lem Bethe degree},
the grading on $\gltt$ induces a grading on
$\B^0$ and $\B$.

\medskip

As subalgebras of $U(\gltt)$, the algebras $\B^0$ and $\B$ act on any
$\gltt$-module $M$. Consider the $\gltt$-module $\V^S$ graded as in
Section \ref{sec grading on V}.

\begin{lem}
The actions of $\B^0$ and $\B$ on $\V^S$ are graded.
\qed
\end{lem}

Denote $\Bv$ (resp. $\Bvz$) the image of the Bethe algebra $\B$
(resp. $\B^0$) in $\End(\V^S)$.

\begin{lem}
\label{Uz}
Each of the Bethe algebras $\Bv$ and $\Bvz$
contains the algebra of operators of multiplication
by elements of\/ $\C[z_1,\dots,z_n]^S$.
\end{lem}
\begin{proof}
An element $B_{1j}=B_{1j}^0=
e_{11}\otimes t^{j-1} + e_{22}\otimes t^{j-1}$
acts on $\V^S$ as the operator
of multiplication by $\sum_{s=1}^n z_s^{j-1}$.
\end{proof}

For $i=1,\dots,n$, let $\sigma_i$ denote the $i$-th elementary symmetric
function of $z_1,\dots,z_n$. We have $\C[\sigma_1,\dots,\sigma_n]=
\C[z_1,\dots,z_n]^S$. The embeddings in Lemma \ref{Uz} of
$\C[\sigma_1,\dots,\sigma_n]$ to $\Bv$ and $\Bvz$ provide
$\Bv$ and $\Bvz$ with structures of
$\C[\sigma_1,\dots,\sigma_n]$-modules.

\subsection{Weight, isotypical and graded decompositions of $\V^S$}
As a $\C[z_1,\dots,z_n]^S$-module, $\V^S$ has the form
\beq
\label{glt-C^S module}
\V^S\ \simeq \ V^{\otimes n} \otimes \C[z_1,\dots,z_n]^S\ .
\eeq
This is an isomorphism of $\glt$-modules, if $\glt$
acts on $\C[z_1,\dots,z_n]^S$ trivially
and acts on $V^{\otimes n}$
in the standard way.

The $\glt$-weight decomposition of $\V^S$ has the form
\vvn.3>
\beq
\label{weight dec}
\V^S\ = \
\oplus_{m=0}^n \V^S[n-m,m]\ \simeq \
\oplus_{m=0}^n V^{\otimes n}[n-m,m]
\otimes \C[z_1,\dots,z_n]^S\ .
\vv.3>
\eeq
We say that a weight $(n-m,m)$ is lower than a weight $(n-m',m')$ if
$n-m<n-m'$.

\medskip
Consider the decomposition of $\V^S$ into isotypical components of the
$\glt$-action,
\vvn.3>
\beq
\label{iso decom}
\V^S\ = \ \oplus_{\bs\la} \,
\V^{S,0}_{\bs\la}\
\simeq \
\oplus_{\bs\la} \,
(V^{\otimes n})_{\bs\la} \otimes \C[z_1,\dots,z_n]^S\ ,
\vv.3>
\eeq
where $\V^{S,0}_{\bs\la}$, $(V^{\otimes n})_{\bs\la}$
are the isotypical components corresponding
to the irreducible polynomial $\glt$-module
with highest weight $\bs\la = (n-k,k)$.

\medskip
The graded decomposition of $\V^S$ has the form
\vvn.3>
\beq
\label{gr dec}
\V^S\ = \
\oplus_{j=-n}^\infty (\V^S)_j \ .
\vv.3>
\eeq

Decompositions \Ref{weight dec}, \Ref{iso decom} and \Ref{gr dec} are
compatible. Namely, we can choose a graded basis $v_i, i\in I$, of the
$\C[z_1,\dots,z_n]^S$-module $\V^S$ which agrees with decompositions
\Ref{weight dec}, \Ref{iso decom}, \Ref{gr dec}. That means that each basis
vector $v_i$ lies in one summand of each of decompositions \Ref{weight dec},
\Ref{iso decom}, \Ref{gr dec}.

\begin{lem}
\label{lem on char of VS}
For any $\bs\la =(n-k,k)$, the graded character of $\V^{S,0}_{\bs\la}$
is given by the formula
\beq
\label{char VS0}
\ch_{\V^{S,0}_{\bs\la}}(q)\,=\,
\frac{(1-q^{n-2k+1})^2}
{1-q}\,
\frac{ 1 }
{(q)_{n-k+1}(q)_{k}}\,q^{2k-n}\ .
\eeq
\end{lem}

The lemma follows from Lemma \ref{weyl}.

\medskip

Decomposition \Ref{iso decom} of $\V^S$ into $\glt$-isotypical
components is preserved by the action of $\B^0$. By formula~\Ref{Bii},
for any $\bs\la=(n-k,k)$, the summand
$\V^{S,0}_{\bs\la}$ is the eigenspace of the operator
$B_{22}^0$ with the eigenvalue $k(n-k+1)$.
Hence
\beq
\label{direc B^0}
\Bvz\ = \ \oplus_{\bs\la} \, \B^0_{\bs\la}\ ,
\eeq
where $\B^0_{\bs\la}$ is the image of $\B^0$ in $\End(\V^{S,0}_{\bs\la})$.

\begin{lem}
\label{direct B^0 lem}
The image $\B^0_{\bs\la}$
of $\B^0$ in $\End(\V^{S,0}_{\bs\la})$ is canonically isomorphic to the
image of $\B^0$ in $\End(\sing \V^{S,0}_{\bs\la})$, where
$\sing \V^{S,0}_{\bs\la} \subset \V^S$ is the subspace of singular vectors of weight
$\bs\la$.
\end{lem}

The lemma follows from Schur's lemma.

By \cite{MTV3} the graded character of $\B^0_{\bs\la}$ is given by the
formula
\beq
\label{char B^0_{la}}
\ch_{\B^0_{\bs\la}}(q)\,=\,
\frac{ 1-q^{n-2k+1} }
{(q)_{n-k+1}(q)_{k}}\,q^{2k-n}\ .
\eeq

\subsection{Algebra $\A_{n-2k}\otimes \B_{\bs\la}^0$ and its module
$\V^{S,0}_{\bs\la}$}
\label{Algebra Adotimes Bbsla}
Given an integer $d$, let $\A_d\,=\,\C[b]/\langle b^{d+1}\rangle$\,.
The algebra
$\A_{n-2k}\otimes \B^0_{\bs\la}$ acts on $\V^{S,0}_{\bs\la}$ by the rule,
\bea
b^j\otimes B\ \mapsto \ e_{21}^jB
\eea
for any $j$ and $B\in \B^0_{\bs\la}$.
Define a grading on $\A_{n-2k}\otimes \B^0_{\bs\la}$ by setting
$\deg\,(b^j\otimes B)\,=\, -j+\deg\,B$. The action of
$\A_{n-2k}\otimes \B^0_{\bs\la}$ on $\V^{S,0}_{\bs\la}$ is graded.

\subsection{Deformed isotypical components of $\V^S$}
\label{Deformed isotypical components}
In this section we obtain a decomposition of the algebra
$\Bv$ similar to decomposition
\Ref{direc B^0} of the algebra $\Bvz$.

For $\bs\la=(n-k,k)$, denote
$\V^S_{\bs\la}\subset \V^S$ the generalized eigenspace of the operator
$B_{22}\in \B$ with the eigenvalue $k(n-k+1)$.
Clearly, $\V^S_{\bs\la}$ is a
$\C[z_1,\dots,z_n]^S$-submodule.

\begin{lem}
\label{lem on isotyp decomp}
We have the following three properties.
\begin{enumerate}
\item[(i)]

Consider a graded basis $v_i,i\in I$,
of the free $\C[z_1,\dots,z_n]^S$-module $\V^S$ which agrees with
decompositions \Ref{weight dec}, \Ref{iso decom}, \Ref{gr dec},
see Section \ref{VS}.
Let a subset $I_{\bs\la}\subset I$ be such that
the vectors $v_i,i\in I_{\bs\la}$, form a basis of
$\V^{S,0}_{\bs\la}$. Then the
$\C[z_1,\dots,z_n]^S$-module
$\V^{S}_{\bs\la}$ has a basis
$w_i,i\in I_{\bs\la}$, such that for all $i$, we have $\deg w_i = \deg v_i$ and
$w_i = v_i + v_i'$, where $v_i'$ lies in the sum of the $\glt$-weight components
of $\V^S$ of weight
lower than the weight of $v_i$.

\item[(ii)]

We have
\beq
\label{dir dec}
\V^S\ =\ \oplus_{\bs\la}\,\V^{S}_{\bs\la}\ .
\eeq
\item[(iii)]

$\V^{S}_{\bs\la}$
is a graded free
$\C[z_1,\dots,z_n]^S$-module
of rank equal to the rank of the isotypical component
$\V^{S,0}_{\bs\la}$. The graded character of $\V^{S}_{\bs\la}$ is given by the formula
\beq
\label{char VS}
\ch_{\V^{S}_{\bs\la}}(q)\,=\,
\frac{(1-q^{n-2k+1})^2}
{1-q}\,
\frac{ 1 }
{(q)_{n-k+1}(q)_{k}}\,q^{2k-n}\ .
\eeq

\end{enumerate}
\end{lem}
\begin{proof}
The operator $B_{22} : \V^S \to \V^S$ is of degree zero.
The matrix $g=(g_{ij})$ of $B_{22}$ in the basis
$v_i,i\in I$, has entries in $\C[z_1,\dots,z_n]^S$.
By \Ref{nilp formula}, the matrix $g$ is lower triangular with
the diagonal entries $g_{ii}=k(n-k+1)$ for all $i\in I_{\bs\la}$.

The corresponding
generalized eigenspace $\V^S_{\bs\la}\subset \V^S$ of $B_{22}$
is the kernel of
the matrix $(g-k(n-k+1))^d$
for a suitable large integer $d$.
The kernel of such a matrix
has properties (i-iii).
\end{proof}

It is clear that $\V^{S}_{\bs\la}\subset \V^S$ are $\B$-submodules.
We call the $\B$-modules $\V^{S}_{\bs\la}$
{\it the deformed isotypical components}.

We have
\beq
\label{direc B}
\Bv\ = \ \oplus_{\bs\la} \, \B_{\bs\la}\ ,
\eeq
where $\B_{\bs\la}$ is the image of $\B$ in $\End(\V^{S}_{\bs\la})$.

\subsection{Epimorphisms
$p_{\bs\la}^\V :\V^S_{\bs\la}\to\sing\,\V^{S,0}_{\bs\la}$ and
$p_{\bs\la}^\B : \B_{\bs\la} \to \B^0_{\bs\la}$}

\label{Epimorphism p}

For $\bs\la=(n-k,k)$, let $\V^S_{\bs\la}$ be the corresponding
deformed isotypical component. Let $v_i, i\in I_{\bs\la},$ be a
basis of the isotypical component $\V^{S,0}_{\bs\la}$, which agrees
with decompositions
\Ref{weight dec}, \Ref{iso decom}, \Ref{gr dec}.
Let $I_{\bs\la,s}\subset I_{\bs\la}$ be the subset such that
the vectors $v_i, i\in I_{\bs\la,s}$, form a basis of the
$\C[z_1,\dots,z_n]^S$-module $\sing\,\V^{S,0}_{\bs\la}$,
where $\sing\,\V^{S,0}_{\bs\la}$ is the submodule of singular vectors.

Let $w_i, i\in I_{\bs\la},$ be a basis of the deformed isotypical
component $\V^{S}_{\bs\la}$, which has properties described in Lemma
\ref{lem on isotyp decomp} with respect to the basis
$v_i, i\in I_{\bs\la}$.

Define a $\C[z_1,\dots,z_n]^S$-module epimorphism
\beq
\label{la epi S}
p_{\bs\la}^\V\ :\ \V^S_{\bs\la}\ \to\ \sing\,\V^{S,0}_{\bs\la}\
\eeq
by the formula:
$w_i \mapsto v_i$ for $i\in I_{\bs\la,s}$ and
$w_i \mapsto 0$ for $i\in I_{\bs\la} \minus I_{\bs\la,s}$.

\begin{lem}
\label{lem on p-la}
We have the following properties.
\begin{enumerate}
\item[(i)]
The kernel of $p^\V_{\bs\la}$ is a $\B$-submodule of the deformed isotypical component
$\V^S_{\bs\la}$ and, therefore, $p^\V_{\bs\la}$ induces
a $\B$-module structure on $\sing\,
\V^{S,0}_{\bs\la}\simeq \V^S_{\bs\la}/({\rm ker}\,p^\V_{\bs\la})$.

\item[(ii)] For this $\B$-module structure on $\sing\,
\V^{S,0}_{\bs\la}$, the image of the $\B$ in $\End (\sing\,
\V^{S,0}_{\bs\la})$ is canonically isomorphic to the image of $\B^0$ in
$\End (\sing\, \V^{S,0}_{\bs\la})$. More precisely, for every
$(i,j)$, the elements $B_{ij}\in \B$ and $B^0_{ij}\in \B^0$ have the
same image.

\end{enumerate}
\end{lem}

\begin{proof}
Lemma follows from Lemma \ref{lem on isotyp decomp}, formula
\Ref{nilp formula} and Theorem \ref{T-thm}.
\end{proof}

By Lemmas \ref{direct B^0 lem}
and
\ref{lem on p-la}, the epimorphism $p^\V_{\bs\la}$ determines an
algebra epimorphism
\beq
\label{p-la-b}
p_{\bs\la}^\B
\ :\ \B_{\bs\la} \to \B^0_{\bs\la}\ .
\eeq
It is clear
$p_{\bs\la}^\B$ is graded and
$p_{\bs\la}^\B$ is a homomorphism of $\C[\sigma_1,\dots,\sigma_n]$-modules.

\medskip

\section{More on deformed isotypical components}
\label{sec More}

\subsection{Deformed isotypical components of $\Ma$}
\label{sec Deformed Ma}

Given a sequence of complex numbers $\bs a=(a_1,\dots,a_n)\in\C^n$,
consider the $\gltt$-module $\V^S/{I^\V_{\bs a}}$ as in Section
\ref{Weyl modules as quotients}.
As a $\glt$-module, $\V^S/{I^\V_{\bs a}}$ is isomorphic to
$V^{\otimes n}$ by Lemma \ref{factor=weyl}.

Consider the $\glt$-weight decomposition
of $\V^S/{I^\V_{\bs a}}$
and its decomposition into
$\glt$-isotypical components, respectively,
\begin{align}
\label{both decomp}
\V^S/I^\V_{\bs a} & \,{}=\, \oplus_{m=0}^n\,(\V^S/I^\V_{\bs a})[n-m,m]\ ,
\\
\V^S/I^\V_{\bs a}& \,{}=\, \oplus_{\bs\la}\,(\V^S/{I^\V_{\bs a}})_{\bs\la}\ .
\notag
\end{align}
Consider a graded basis $v_i,i\in I$,
of the free $\C[z_1,\dots,z_n]^S$-module $\V^S$ which agrees with
decompositions \Ref{weight dec}, \Ref{iso decom}, \Ref{gr dec}.
This basis induces a $\C$-basis
$\bar v_i,i\in I$, of
$\V^S/I^\V_{\bs a}$,
which agrees with both decompositions in
\Ref{both decomp}. For any $\bs\la$, the vectors $\bar v_i,i\in
I_{\bs\la}$, form a weight basis of the isotypical component
$(\V^S/{I^\V_{\bs a}})_{\bs\la}$.

\medskip

Denote
\bea
\Ma\ = \ \V^S/{I^\V_{\bs a}}\ .
\eea
For $\bs\la=(n-k,k)$, denote
\bea
\Mal\ \subset\ \Ma
\eea
the generalized eigenspace of the operator
$B_{22}\in \B$ with eigenvalue $k(n-k+1)$.

Lemma \ref{lem on isotyp decomp} has the following analog.

\begin{lem}
\label{lem on isotyp decomp M}
We have the next three properties.
\begin{enumerate}
\item[(i)]
$\Mal$ is a $\C$-vector space of the dimension
equal to the dimension of
$(\V^S/{I^\V_{\bs a}})_{\bs\la}$.

\item[(ii)]
$\Mal$ has a basis
$w_i,i\in I_{\bs\la}$, such that for all $i$,
$w_i = \bar v_i + v_i'$ where $v_i'$ lies in the sum of the
$\glt$-weight components
of $\Ma$ of weight
lower than the weight of $\bar v_i$.
\item[(iii)] We have
\beq
\label{dir dec M}
\Ma\ =\ \oplus_{\bs\la}\,\Mal\ .
\eeq
\end{enumerate}
\end{lem}

It is clear that the subspaces $\Mal\subset \Ma$ are $\B$-submodules.
We call the $\B$-modules $\Mal$
{\it the deformed isotypical components} of $\Ma$.

\subsection{Bethe eigenleaves}
\label{ Bethe eigenleaves}
Let $\phi :\B^0 \to \C$ be a homomorphism.
Let $W_\phi\subset\Ma$ be the generalized
eigenspace of the $\B^0$-action with eigenvalue $\phi$.
Since the $\B^0$-action commutes with the $\glt$-action,\
$W_\phi$ is a $\glt$-submodule. Assume that $W_\phi$ is an irreducible
$\glt$-module with highest weight $\bs\la=(n-k,k)$.
This means, in particular, that
$Bw=\phi(B)w$ for all $w\in W_\phi$ and $B\in\B^0$.

Choose a weight basis $u_i, i=0,\dots, n-2k$, of $W_\phi$. Choose a
finite set $B^0_{ij}, (i,j)\in J$, of the standard generators of
$\B^0$, such that $W_\phi$ is the common generalized eigenspace of the
operators $B^0_{ij}\in\B^0, (i,j)\in J$, with eigenvalues
$\phi(B^0_{ij})$, respectively.

\medskip
Under these assumptions, denote $\Malp\subset\Ma$
the generalized eigenspace of the operators
$B_{ij}\in \B$, $(i,j)\in J$, with eigenvalues
$\phi(B^0_{ij})$, respectively.

Lemmas \ref{lem on isotyp decomp} and
\ref{lem on isotyp decomp M} have the following analog.

\begin{lem}
\label{lem on isotyp decomp W}
Under these assumptions, we have the next two properties.
\begin{enumerate}
\item[(i)]
$\Malp$ is a $\C$-vector subspace of $\Mal$ of the dimension
equal to the dimension of
$W_\phi$.

\item[(ii)]
$\Malp$ has a basis
$w_i,i\in 0,\dots,n-2k$, such that for all $i$,
$w_i = u_i + u_i'$ where $u_i'$ lies in the sum of the
$\glt$-weight components
of $\Mal$ of weight
lower than the weight of $u_i$.
\end{enumerate}
\end{lem}

It is clear that $\Malp\subset \Mal$ is a $\B$-submodule.
We call the $\B$-module $\Malp$
{\it a Bethe eigenleaf} of $\Mal$.

\begin{lem}
\label{lem sum of leaves}
Let $\bs a\in \R^n$ be such that all roots of the polynomial
$u^n+\sum_j(-1)^ja_ju^{n-j}$ are distinct and real. Then
the $\B$-module $\Ma$ is the direct sum of its Bethe eigenleaves,
\beq
\label{sum of leaves}
\Ma\ = \ \sum_{\phi,\bs \la}\,\Malp\ .
\eeq
\end{lem}

\begin{proof}
Denote $ \sing\,\Ma \,=\, \{v\in \Ma\ | \ e_{21}v=0 \}$ the
subspace of singular vectors. By \cite{MTV3}, the action of $\B^0$ on
$ \sing\,\Ma$ has simple spectrum if all roots of the polynomial
$u^n+\sum_j(-1)^ja_ju^{n-j}$ are distinct and real. This fact and
property \Ref{nilp formula} imply the lemma.
\end{proof}

\subsection{The universal differential operator of $\V^S$}
\label{The universal differential operator of VS}

\begin{lem}[cf. Lemma 5.9 in \cite{MTV3}]
\label{elm on diff oper}
Denote $\D_{\V^S}$ the universal differential operator of
the $\B$-module $\V^S$. Then $\D_{\V^S}$ has the form
\beq
\D_{\V^S}\ =\ \der^2 - \bar B_1(u)\partial + \bar B_2(u)\ ,
\eeq
where
\vvn-.5>
\begin{gather*}
\bar B_1(u)\ = \ \frac{W'(u)}{W(u)}\ ,\qquad
\bar B_2(u)\ = \ \frac{U(u)}{W(u)}\ ,
\\[6pt]
W(u)\ = \ \prod_{i=1}^n(u-z_i)\ ,
\qquad
U(u)\ = \ \sum_{i=1}^{n}\,U_iu^{n-i} \ ,
\end{gather*}
with $U_i \in \End_{\C[z_1,\dots,z_n]^S}(\V^S)$ and
\beq
\label{U_0}
\rlap{\hss\rightline{\hfill\qedph
$U_1\ =\ B_{21}\ = \ \sum_{s=1}^n \,e_{21}^{(s)}\ .$\qed}}
\eeq
\end{lem}

\subsection{The universal differential operator of $\Ma$}
\label{The universal differential operator of Ma}

\begin{lem}
\label{elm on diff oper Ma}
Let $\D_{\Ma}$ be the
universal differential operator of the $\B$-module
$\Ma$ and $y(u)$ an $\Ma$-valued function of $u$. Then
all solutions to the differential equation
$\D_{\Ma}y(u) = 0$ are $\Ma$-valued polynomials.
\end{lem}

\begin{proof}
By Theorem 8.4 in \cite{MTV2}, every solution is a linear combination
of the functions of the form $e^{c u}p(u)$, where $p(u)$ is an
$\Ma$-valued polynomial and $c\in\C$. Writing $\D_{\Ma}e^{c u}p(u)=0$
and computing the leading term, we conclude that $c=0$.
\end{proof}

\subsection{The universal differential operator of a Bethe eigenleaf}
\label{operator of leaf}

\begin{lem}
\label{lem on leaf}
Let $\bs\la=(n-k,k)$.
Let $\Malp$ be a Bethe eigenleaf.
Then the universal differential operator $\D_{\Malp}$ of the $\B$-module
$\Malp$ has the form
\vvn.4>
\beq
\label{D-Malp}
\D_{\Malp}\ =\ \der^2 - \bar B_1(u)\partial + \bar B_2(u)\ ,
\vv.2>
\eeq
where
\vvn-.5>
\begin{gather*}
\bar B_1(u)\ = \ \frac{W'(u)}{W(u)}\ , \qquad
\bar B_2(u)\ = \ \frac{U(u)}{W(u)}\ ,
\\[6pt]
W(u)\ = \ u^n + \sum_{i=1}^n\,(-1)^i a_i u^{n-i}\ ,
\qquad
U(u)\ = \ \sum_{i=1}^{n}\, U_i u^{n-i} \ ,
\end{gather*}
with $U_i \in \End_{\C}(\Malp)$. Moreover,
\be
U_1\ =\ B_{21}|_\Malp\ =\ e_{21}|_\Malp
\ee
and for any $i>1$, we have
\beq
\label{U_i}
U_i\ =\ \sum_{j=0}^{n-2k}\, c_{ij}\,(B_{21}|_\Malp)^j
\eeq
where $c_{ij}\in\C$ and $c_{20}=k(n-k+1)$.
\end{lem}

\begin{proof} We need to prove \Ref{U_i}
and formula $c_{20}=k(n-k+1)$. Everything else follows from Lemma
\ref{elm on diff oper}.

The operators $U_i$ are elements of the Bethe algebra of $\Malp$. The
Bethe algebra of $\Malp$ contains the scalar operators and the nilpotent
operator $B_{21}|_\Malp$. On the complex $n-2k+1$-dimensional vector space
$\Malp$, we have $(B_{21}|_\Malp)^{n-2k} \neq 0$ and $(B_{21}|_\Malp)^{n-2k+1} =
0$. Hence,
every element of that algebra is a polynomial in
$B_{21}|_\Malp$ with complex coefficients. Formula \Ref{U_i} is proved.

Formula $c_{20}=k(n-k+1)$ follows from \Ref{nilp formula} and properties
of the universal differential operator of the algebra $\B^0$ associated with the
isotypical component
$(\V^S/{I^\V_{\bs a}})_{\bs\la}$, see \cite{MTV3}.
\end{proof}

\begin{lem}
\label{lem on solutions on leaf}
Let $\Malp$ be a Bethe eigenleaf and
$\D_{\Malp}$ the universal differential operator of the $\B$-module
$\Malp$, see Lemma \ref{lem on leaf}.
Then all
solutions to the $\Malp$-valued differential equation
$\D_\Malp y(u)=0$ are $\Malp$-valued polynomials.
\end{lem}

The lemma follows from Lemma \ref{elm on diff oper Ma}.

\medskip
To a Bethe eigenleaf $\Malp$, we assign a scalar differential operator
\vvn.3>
\beq
\label{D0-Malp}
\D_{\Malp,0}
\ =\ \der^2 - \frac{W'(u)}{W(u)}
\partial + \frac{\sum_{i=2}^{n} c_{i0}\, u^{n-i}}{W(u)}\ ,
\vv.3>
\eeq
see notation in Lemma~\ref{lem on leaf}. It is clear, that
any solution to the differential equation
$\D_{\Malp,0}\,y(u)=0$ is a polynomial of degree $k$ or $n-k+1$.

Let $w_i,i\in 0,\dots,n-2k$ be a basis of $\Malp$ indicated in Lemma
\ref{lem on isotyp decomp W}. Let $y(u) = \sum_i y_i(u)w_i$ be a
solution to the differential equation
$\D_{\Malp}y(u)=0$, then $y_0(u)$ is a solution to the differential equation
$\D_{\Malp,0}\,y(u)=0$.

\medskip
Let $F_0(u),G_0(u) \in \C[u]$ be polynomials of degrees $k$ and $n-k+1$,
respectively. Then the kernel of the differential operator
\beq
\label{DFG}
\D_{F_0,G_0}
\ =\ \der^2 - \frac{\Wr'(F_0,G_0)}{\Wr(F_0,G_0)}
\partial + \frac{\Wr(F'_0,G'_0)}{\Wr(F_0,G_0)}\
\eeq
is the two-dimensional subspace of $\C[u]$ generated by $F_0(u),G_0(u)$.

\begin{lem}
\label{lem on pairs of polynomials and leaves}
For any generic pair of polynomials $F_0(u),G_0(u) \in \C[u]$
with
$\deg F_0(u)=k,\,\deg G_0(u)=n-k+1$, there exists a unique Bethe eigenleaf
$\Malp$, such that
\linebreak $\D_{\Malp,0}\,=\,\D_{F_0,G_0}$.
\end{lem}

\begin{proof}
For $F_0(u),G_0(u) \in \C[u]$ with
$\deg F_0(u)=k,\,\deg G_0(u)=n-k+1$, define
$\bs a=(a_1,\dots,a_n)$ by the formula
\bea
\Wr(F_0(u),G_0(u))\ = \ (n-2k)\,(u^n + \sum_{j=1}^n(-1)^ja_ju^{n-j})\ .
\eea
By \cite{MTV3},
for every generic pair $F_0(u),G_0(u) \in \C[u]$, there exists a unique
eigenvector $v\in \sing\,\Ma$ of the Bethe algebra $\B^0$
with
\bea
B^0_{ij}v\ =\ c_{ij}v
\eea
for some $c_{ij}\in\C$ and all $(i,j)$, such that
\bea
\D_{F_0,G_0}\ =\
\der^2 - \sum_{j}c_{1j}u^{-j}
\partial + \sum_{j}c_{2j}u^{-j}\ .
\eea
This fact and
property \Ref{nilp formula} imply the lemma.
\end{proof}

\section{Algebra $\Ol$}
\label{sec Ol}

\subsection{Wronskian conditions}
\label{Wronskian conditions II}
Fix nonnegative integers $k$ and $d$. Define an algebra
$\A_d = \C[b]/\langle b^{d+1}\rangle$, with $b$ a generator of $\A_d$.
Consider the expressions:
\begin{align}
\label{new f,g}
f(u) & {}\,=\, \sum_{i=0}^{k-1}f_iu^i + u^k +
\sum_{i=1}^{d} \tilde f_{k+i}b^iu^{k+i}\ ,
\\[4pt]
g(u) & \,{}=\, \sum_{i=0}^{k-1}g_iu^i + \sum_{i=k+1}^{k+d}g_iu^i + u^{k+d+1} +
\sum_{i=1}^{d} \tilde g_{k+d+1+i}b^iu^{k+d+1+i}\ .
\notag
\end{align}
These are polynomials in $u, f_i, g_i, \tilde f_{k+i},\tilde g_{k+d+1+i}$
with coefficients in $\A_d$.

Consider the polynomials
\beq
\label{f and g}
\Wr(f(u),g(u)) = \sum_{j=0}^{2k+3d} U_ju^j\ ,
\qquad
\Wr(f'(u),g'(u))= \sum_{j=0}^{2k+3d-2} V_ju^j\ ,
\eeq
where $U_j,V_j$ are suitable polynomials in
$ f_i, g_i,\,\tilde f_{k+i}b^i,\, \tilde g_{k+d+1+i}b^i$\,
with integer coefficients.

It is easy to see that
\bea
U_i\,=\,0 \ , \qquad V_{i-2}\,=\,0\ , \qquad {\rm for}\ i> 2k+2d \ ,
\eea

\begin{thm}
\label{thm on algebra II}
Consider the system of $2d$ equations
\begin{align}
\label{main eqns II}
U_{2k+d+1} & \,{}=\, 0\ , \qquad V_{2k+d-2+1}- U_{2k+d}b = 0\ ,
\\
U_{2k+d+i} & \,{}=\,0\ ,\qquad V_{2k+d-2+i} = 0\ ,\qquad{ for}
\ {}\ i=2,\dots,d\ ,\
\notag
\end{align}
with respect to
$\tilde f_{k+i}b^i, \tilde g_{k+d+1+i}b^i$, $i=1,\dots,d$.
Then there exist $2d$ polynomials
$\tilde\phi_{k+i}$, $\tilde\psi_{k+d+1+i}$ in $2k+d$ variables
\vvn.3>
\beq
\label{f,g}
f_i,\ i=0,\dots,k-1, \quad {\rm and} \quad g_i,\ i=0,\dots,k-1,k+1,\dots,k+d,
\vv.3>
\eeq
with coefficients in $\A_d$, such that system \Ref{main eqns II} is equivalent
to the system of $2d$ equations:
\vvn.3>
\beq
\label{main eqns II add}
\tilde f_{k+i} b^i\ =\ \tilde\phi_{k+i}\ , \qquad
\tilde g_{k+d+1+i} b^i\ =\ \tilde\psi_{k+d+1+i}\ , \qquad i=1,\dots,d\ .
\vv.3>
\eeq
\end{thm}

\medskip
Let $\mc E$ be a $\C$-algebra. Abusing notation,
we will write $b^jy$ instead of $b^j \otimes y\in \A_d\otimes\mc E$
for any $0\leq j\leq d$ and $y\in\mc E$.

We denote by $\C[\{f,g,\tilde f,\tilde g\}]$ the polynomial algebra in
all variables $f_i,\,g_i,\,\tilde f_{k+i}, \,\tilde g_{k+d+1+i}$ appearing
in~\Ref{new f,g}, and by $\C[\{f,g\}]$ the polynomial algebra of all variables
$f_i,g_i$ described in \Ref{f,g}.

\medskip
Let
\beq
\label{Ckd}
\mc C_{k,d}\,\subset\,\A_d\otimes
\C[\{f,g,\tilde f,\tilde g\}]
\eeq
be the $\C$-subalgebra generated by all elements
$1,\,f_i,\,g_i,\,\tilde f_{k+i}b^i, \,\tilde g_{k+d+1+i}b^i, b$.

\begin{cor}
\label{cor on algebra II}
Consider the ideal $ I$ in $\mc C_{k,d}$ generated by
the left hand sides of equations \Ref{main eqns II}. Then the quotient algebra
$\mc C_{k,d}/I$ is canonically isomorphic to the algebra
$\A_d\otimes \C[\{f,g\}]$.
\end{cor}

\medskip
\noindent
{\it Proof of Theorem \ref{thm on algebra II}.}\
The four equation in \Ref{main eqns II} have the following form
\vvn.3>
\begin{gather}
\label{eq 1}
d\tilde f_{k+1}b +
(d+2)\tilde g_{k+d+2}b + Y_{2k+d+1}\,=\,0\ ,
\\[5pt]
d(k+1)(k+d+1)\tilde f_{k+1}b + (d+2)k(k+d+2)\tilde g_{k+d+2}b -
(d+1)b + Z_{2k+d-2+1}\,=\,0\ ,
\notag
\\[-8pt]
\notag
\end{gather}
\begin{align}
\label{eq 2}
(d+1-i) & \tilde f_{k+i}b^i +
(d+1+i)\tilde g_{k+d+1+i}b^i +{}
\\
& \sum_{j=1}^i\,(d+1+i-2j)\tilde f_{k+j}b^j\tilde g_{k+d+1+i-j}b^{i-j} +
Y_{2k+d+i}\,=\, 0\ ,
\notag
\\[6pt]
(d+1-i)(k+i) & (d+k+1)\tilde f_{k+i}b^i +
(d+1+i)(k+i)(d+k+1)\tilde g_{k+d+1+i}b^i +{}
\notag
\\
\sum_{j=1}^i\,(d+1+i &{}- 2j)(k+j)(d+k+1+i-j)
\tilde f_{k+j}b^j\tilde g_{k+d+1+i-j}b^{i-j} + Z_{2k+d-2+i}\,=\,0\ .
\notag
\end{align}
In equations \Ref{eq 1}, $Y_{2k+d+1}$ and $Z_{2k+d-1}$
are suitable polynomials in
the variables $f_j,g_j,$ $\tilde f_{k+j}b^j,$ $ \tilde g_{k+d+1+j}b^j$
such that every monomial of $Y_{2k+d+1}$ and every monomial of $Z_{2k+d-1}$
has degree at least two with respect to $b$.
In equations \Ref{eq 2},
$Y_{2k+d+i}$ and $Z_{2k+d-2+i}$
are suitable polynomials in
the variables $f_j,g_j,\tilde f_{k+j}b^j, \tilde g_{k+d+1+j}b^j$
such that every monomial of $Y_{2k+d+i}$ and
every monomial of $Z_{2k+d-2+i}$ has degree at least $i+1$ with respect to $b$.

Transforming equations \Ref{main eqns II} to
equations \Ref{eq 1} and \Ref{eq 2} we distinguished the leading terms
(with respect to powers of $b$)
of the polynomials in \Ref{main eqns II}.

The variables
$\tilde f_{k+1}b,\, \tilde g_{k+d+2}b$ enter linearly the two equations
in \Ref{eq 1}. The determinant of this $2\times 2$
system is nonzero.
Solving this linear system, gives
\begin{align}
\label{eq 3}
\tilde f_{k+1}b & \,{}=\, c_{k+1}b + W_{k+1}\ ,
\\
\tilde g_{k+d+2}b & \,{}=\, c_{k+d+2}b + W_{k+d+2}\ ,
\notag
\end{align}
where $c_{k+1},c_{k+d+2}\in \C$ and
$W_{k+1}, W_{k+d+2}$
are suitable polynomials in
the variables $f_j,g_j,$ $\tilde f_{k+j}b^j,$ $ \tilde g_{k+d+1+j}b^j$
such that every monomial of $W_{k+1}$ and every monomial of $W_{k+d+2}$
has degree at least two with respect to $b$.

Consider the two equations of \Ref{eq 2} corresponding to $i=2$,
\begin{gather}
\label{eq 4}
a_1\tilde f_{k+2}b^2 +
a_2\tilde g_{k+d+3}b^2
+ a_3\tilde f_{k+1}b\tilde g_{k+d+2}b + Y_{2k+d+2}\,=\,0\ ,
\\
b_1\tilde f_{k+2}b^2 + b_2\tilde g_{k+d+3}b^2
+ b_3\tilde f_{k+1}b\tilde g_{k+d+2}b + Z_{2k+d} = 0\ ,
\notag
\end{gather}
where the numbers $a_j,b_j$ are determined in \Ref{eq 2}.
It is easy to see that the determinant of the matrix
$\left( \begin{array}{clcr}
a_1 & a_2
\\
b_1 & b_2
\end{array} \right)$
is nonzero.
Replace in \Ref{eq 4} the product $\tilde f_{k+1}b\tilde g_{k+d+2}b$ with
\be
(c_{k+1}b + W_{k+1})(c_{k+d+2}b + W_{k+d+2})\ .
\ee
Then solving
the linear system in \Ref{eq 4} with respect to
$\tilde f_{k+2}b^2$, $\tilde g_{k+d+3}b^2$ we get
\begin{align*}
\tilde f_{k+2}\>b^2 & \,{}=\, c_{k+2}\>b^2 + W_{k+2}\ ,
\\
\tilde g_{k+d+3}\>b^2 & \,{}=\, c_{k+d+3}\>b^2 + W_{k+d+3}\ ,
\end{align*}
where $c_{k+2},c_{k+d+3}\in \C$ and
$W_{k+2}, W_{k+d+3}$
are suitable polynomials in
the variables $f_j,g_j,$ $\tilde f_{k+j}b^j,$ $ \tilde g_{k+d+1+j}b^j$
such that every monomial of $W_{k+2}$ and every monomial of $W_{k+d+3}$
has degree at least three with respect to $b$.

Repeating this procedure we obtain for every $i=1,\dots,d$, equations
\begin{align*}
\tilde f_{k+i}b^i & \,{}=\, c_{k+i}b^i + W_{k+i}\ ,
\\
\tilde g_{k+d+1+i}b^i & \,{}=\, c_{k+d+1+i}b^i + W_{k+d+1+i}\ ,
\end{align*}
where $c_{k+i},c_{k+d+1+i}\in \C$ and
$W_{k+i}, W_{k+d+1+i}$ are
suitable polynomials in
the variables $f_j,g_j,$ $\tilde f_{k+j}b^j,$ $ \tilde g_{k+d+1+j}b^j$
such that every monomial of $W_{k+i}$ and every monomial of $W_{k+d+1+i}$
has degree at least $i+1$ with respect to $b$.

\medskip
For every $m$, replace in $W_m$ every variable
$\tilde f_{k+j}b^j$ and $\tilde g_{k+d+1+j}b^j$ with
$c_{k+j}b^j + W_{k+j}$ and $c_{k+d+1+j}b^j + W_{k+d+1+j}$, respectively.
Then for every $i=1,\dots,d$, we have
\begin{align}
\label{eq 6}
\tilde f_{k+i}b^i & \,{}=\, X^1_{k+i} + X^2_{k+i}\ ,
\\
\tilde g_{k+d+1+i}b^i & \,{}=\, X^1_{k+d+1+i} + X^2_{k+d+1+i}\ ,
\notag
\end{align}
where $X^1_{k+i}$, $X^1_{k+d+1+i}$ are suitable polynomials
in the $k+2d$ variables $f_j, g_j$, and $X^2_{k+i}$,
$ X^2_{k+d+1+i}$
are suitable polynomials in the variables
$f_j,g_j,$ $\tilde f_{k+j}b^j,$ $ \tilde g_{k+d+1+j}b^j$
such that every monomial of $X^2_{k+i}$ and every monomial of
$X^2_{k+d+1+i}$
has degree at least $i+2$ with respect to $b$.

Iterating this procedure we prove the theorem.
\qed

\subsection{Algebras $\Ol$ and $\Olo$}

For given $\bs\la=(k+d,k)$, we define an algebra $\Ol$ by the formula
\be
\Ol\,=\,\mc C_{k,d}/I\;,
\ee
where $\mc C_{k,d}$ is defined in~\Ref{Ckd}.
For any $x\in\mc C_{k,d}$, its image in $\Ol$ will be denoted $\{x\}$. Let
\be
\mc O^0_{\bs\la}\ =\ \C[\{f,g\}]\ .
\ee
By Corolalry~\ref{cor on algebra II}, the algebra homomorphism
\be
\label{iso formula}
q_{\bs\la}\ :\ \A_d \otimes \Olo\ \to\ \Ol\ ,
\qquad
f_i \mapsto \{f_i\}, \
g_i \mapsto \{g_i\}, \
b \mapsto \{b\}\ ,
\ee
for all $i$, is an isomorphism.

\medskip
Introduce the polynomials $\{f\}(u),\{g\}(u)\in \Ol[u]$ by the formulae:
\begin{align}
\label{f g}
\{f\}(u) & \,{}=\, \sum_{i=0}^{k-1}\{f_i\}u^i + u^k +
\sum_{i=1}^{d} \{\tilde f_{k+i}b^i\}u^{k+i}\ ,
\\
{} \{g\}(u) & \,{}=\,\sum_{i=0}^{k-1}\{g_i\}u^i +
\sum_{i=k+1}^{k+d}\{g_i\}u^i + u^{k+d+1} +
\sum_{i=1}^{d} \{\tilde g_{k+d+1+i}b^i\}u^{k+d+1+i}\ .
\notag
\end{align}
The polynomials $\{f\}(u), \{g\}(u)$ lie in the kernel of the
differential operator
\beq
\label{diff oper}
\D_\Ol \ =\ \partial^2\ -\ \frac{\Wr'(\{f\},\{g\})}{\Wr(\{f\},\{g\})}\,\partial
\ +\ \frac{\Wr(\{f\}',\{g\}')}{\Wr(\{f\},\{g\})}\ .
\eeq
The operator $\D_\Ol$
will be called the
{\it universal differential operator associated with}
$\Ol$.

\begin{cor}
\label{cor on diff eqn}
In formula \Ref{diff oper},
$\Wr(\{f\},\{g\})$ is a polynomial in $u$ of degree $2k+d$,\
$\Wr(\{f\}',\{g\}')$ is a polynomial in $u$
of degree $2k+d-1$ and the residue at $u=\infty$
of the ratio
\linebreak
$\Wr(\{f\}',\{g\}')/\Wr(\{f\},\{g\})$ equals $\{b\}$.
\end{cor}

\medskip
Introduce a notation for the coefficients of the universal differential operator
$\D_\Ol$:
\beq
\label{F-i}
F_1(u)\ = \frac{\Wr'(\{f\},\{g\})}{\Wr(\{f\},\{g\})}\ ,
\qquad
F_2(u)\ = \ \frac{\Wr(\{f\}',\{g\}')}{\Wr(\{f\},\{g\})}\ .
\eeq
Expand the coefficients in Laurent series at $u=\infty$:
\beq
\label{F-ii}
F_1(u)\ = \
\sum_{j=1}^\infty \,{F_{1j} u^{-j}}\ ,
\qquad
F_2(u)\ = \
\sum_{j=1}^\infty \, {F_{2j} u^{-j}}\ ,
\eeq
where $F_{sj}\in \Ol$,\ $F_{11}= 2k+d,\, F_{21}= \{b\}$.

\begin{lem}
\label{lem on coeff}
The $\C$-algebra $\Ol$ is generated by the elements $F_{sj}$,\ $s=1,2,\,
j=1,2,\dots\,.$
\end{lem}

\begin{proof}
By Theorem \ref{thm on algebra II}, we have an isomorphism
$q_{\bs\la} : \A_d \otimes \Olo\to \Ol$.
Hence, for all $(s,j)$,
we can write
$
F_{sj} = \sum_{t=0}^d F_{sj}^t \{b\}^t,
$
where $F_{sj}^t$ are polynomials in the generators $\{f_i\}$, $\{g_i\}$.
The operator
\bea
\partial^2\ -\ \sum_{j=1}^\infty \, {F_{1j}^0 u^{-j}}\partial
\ +\
\sum_{j=2}^\infty \, {F_{2j}^0 u^{-j}}\
\eea
annihilates the polynomials
$ \{f_0\}+\dots + \{f_{k-1}\}u^{k-1} + u^k$ and
$\{g_0\} + \{g_1\}u+\dots + \{g_{k-1}\}u^{k-1} +
\{g_{k+1}\}u^{k+1} + \dots + \{g_{k+d}\}u^{k+d} + u^{k+d+1}$.
By Lemma 3.3 in \cite{MTV3}, every $\{f_m\}, \{g_m\}$ can be written as a polynomial in
$F_{sj}^0$,\ $s=1,2,\,j=2,3,\dots $, with coefficients in $\C$:
\bea
\{f_m\}\ = \ \phi^0_m(F^0_{sj})\ ,
\qquad
{}\{g_m\}\ = \ \psi^0_m(F^0_{sj})\ .
\eea
We have
\bea
\{f_m\}\ = \ \phi^0_m(F_{sj}) + (\phi^0_m(F^0_{sj}) - \phi^0_m(F_{sj}))
= \phi^0_m(F_{sj}) + \{b\} \phi_m^1\ ,
\\
{}\{g_m\}\ = \ \psi^0_m(F_{sj}) + (\psi^0_m(F^0_{sj}) - \psi^0_m(F_{sj}))
= \psi^0_m(F_{sj}) + \{b\} \psi_m^1\ ,
\eea
where $\phi_m^1, \psi_m^1 \in \Ol$.
These formulae give a presentation of the
elements $\{f_m\}, \{g_m\}$ in terms of $F_{sj}$
modulo the ideal $\langle\{b\}\rangle\subset\Ol$.

Elements $\phi_m^1, \psi_m^1$ can be written as polynomials in the generators
$\{f_i\}, \{g_i\}$ with coefficients in $\C[\{b\}]$:
\bea
\phi^1_m\ =\ \phi_m^1(\{f_i\}, \{g_i\})\ ,
\qquad
\psi_m^1\ = \ \psi_m^1(\{f_i\},\{g_i\})\ .
\eea
Then
\bea
\phi^1_m\ &=&\ \phi_m^1(\{f_i\}, \{g_i\}) =
\phi_m^1(\phi_i^0(F^0_{sj}),\psi_i^0(F^0_{sj}))
\\
&=&
\phi_m^1(\phi_i^0(F_{sj}),\psi_i^0(F_{sj})) +
(\phi_m^1(\phi_i^0(F^0_{sj}),\psi_i^0(F^0_{sj})) -
\phi_m^1(\phi_i^0(F_{sj}),\psi_i^0(F_{sj})) )
\\
&=&
\phi_m^1(\phi_i^0(F_{sj}),\psi_i^0(F_{sj})) +
\{b\}\phi_m^2
\eea
and
\bea
\psi^1_m\ &=&\ \psi_m^1(\{f_i\}, \{g_i\}) =
\psi_m^1(\phi_i^0(F^0_{sj}),\psi_i^0(F^0_{sj}))
\\
&=&
\psi_m^1(\phi_i^0(F_{sj}),\psi_i^0(F_{sj})) +
(\psi_m^1(\phi_i^0(F0_{sj}),\psi_i^0(F^0_{sj})) -
\psi_m^1(\phi_i^0(F_{sj}),\psi_i^0(F_{sj})) )
\\
&=&
\psi_m^1(\phi_i^0(F_{sj}),\psi_i^0(F_{sj})) +
\{b\}\psi_m^2\ ,
\eea
where $\phi_m^2, \psi_m^2\in \Ol$. Thus,
\bea
\{f_m\}
& = & \phi^0_m(F_{sj}) + \{b\} \phi_m^1(\phi_i^0(F_{sj}),\psi_i^0(F_{sj})) +
\{b\}^2\phi_m^2\ ,
\\
{}\{g_m\}
&=& \psi^0_m(F_{sj}) + \{b\}\psi_m^1(\phi_i^0(F_{sj}),\psi_i^0(F_{sj})) +
\{b\}^2\psi_m^2\ .
\eea
These formulae give a presentation of elements $\{f_m\}, \{g_m\}$ in terms of $F_{sj}$
modulo the ideal $\langle\{b\}^2\rangle\subset\Ol$.
Continuing this procedure we prove the lemma.
\end{proof}

Define an algebra epimorphism
\beq
\label{pOl}
p^\O_{\bs\la}\ :\ \Ol\ \to \Olo
\eeq
by the formulae $\{b\} \mapsto 0,\,\{f_i\} \mapsto f_i,\,
\{g_i\} \mapsto g_i$ for all $i$.
Define an algebra monomorphism
\beq
\label{iOl}
i^\O_{\bs\la}\ :\ \A_d\ \to \Ol
\eeq
by the formula $b\mapsto \{b\}$.

\subsection{Grading on $\Ol$ and $\Olo$}

Define
the degrees of the elements $u, b, f_i, g_i, \tilde f_{k+i}b^i,
\tilde g_{k+d+1+i}b^i$ to be
$1,-1,k-i, k+d+1-i, -i, -i$, respectively.
Then the polynomials $f(u), g(u)$, defined in
\Ref{new f,g}, are homogeneous of degree $k$, $k+d+1$, respectively.

Equations of system \Ref{main eqns II} are homogeneous. Hence $\Ol$
has an induced grading. The same rule defines a grading on $\Olo$.
The isomorphism
$q_{\bs\la} : \A_d \otimes \Olo\ \to \Ol$ and
epimorphism $p^\O_{\bs\la} : \Ol \to \Olo$
are graded.

\begin{lem}
\label{char OL}
The graded character of $\Ol$ and $\Olo$ are given by the following formulae:
\vvn.5>
\begin{align}
\label{ch Ol}
\ch_\Ol(q) &\;{}=\; \frac{(1-q^{d+1})^2}{1-q}\,
\frac{q^{-d}}{(q)_{k+d+1}(q)_{k}}\;=\;
\frac{(1-q^{n-2k+1})^2} {1-q}\, \frac{q^{2k-n}}{(q)_{n-k+1}(q)_{k}}\;,
\\[8pt]
\ch_{\Olo}(q) &\; {}=\;\frac{1-q^{d+1} }{(q)_{k+d+1}(q)_{k}}\;=\;
\frac{1-q^{n-2k+1}}{(q)_{n-k+1}(q)_{k}}\;.
\notag
\\[-32pt]
\notag
\end{align}
\vvn.2>
\qed
\end{lem}

Let $F_{ij}\in \Ol$ be the elements defined in \Ref{F-ii}.

\begin{lem}
\label{lem Ol degree}
For any $(i,j)$, the element $F_{ij}$ is homogeneous of degree $j-i$.
\qed
\end{lem}

\section{Special homomorphism of $\Ol$ and Bethe eigenleaves}
\label{Ol and B}

We keep notations of Section \ref{sec Ol}.

\subsection{Differential operators with polynomial kernel}

Let $W(u)\in \C[u]$ be a monic polynomial of degree $2k+d$. Let
$U(u)\in\A_d[u]$ be a polynomial
of the form
\beq
\label{V}
U(u)\ = \
b u^{2k+d-1} \,+\, \sum_{i=0}^{2k+d-2}\sum_{j=0}^dv_{i,j}b^ju^{i}
\eeq
with $v_{ij}\in \C$. Denote
\vvn-.2>
\begin{gather}
\label{diff ex}
\D \ = \ \partial^2\ -\ \frac{W'}{W}\,\partial
\ +\
\frac{U}{W}\ ,
\\[8pt]
\label{chi}
\chi(\al)\ =\ \al(\al-1) - (2k+d)\al + v_{2k+d-2,0}\ ,
\end{gather}
where $\al$ is a variable.

\smallskip
Consider the differential equation
\,$\D y(u)=0$ \,on an $\A_d$-valued function $y(u)$.

\begin{lem}
\label{lem on solutions}
Assume that all solutions to the differential equation $\D y(u)=0$ are polynomials
and $\chi(\al) = (\al-k)(\al-k-d-1)$.
Then the differential equation $\D y(u)= 0$ has unique solutions
$F(u),G(u)$ such that
\vvn.2>
\begin{align*}
F(u) &\;{}=\;
\sum_{i=0}^{k-1}\sum_{j=0}^{d} F_{ij}b^ju^i + u^k +
\sum_{i=1}^{d}\sum_{j=i}^{d}F_{k+i,j}b^ju^{k+i}\ ,
\\[2pt]
G(u) & \;{}=\; \sum_{i=0}^{k-1}\sum_{j=0}^{d} G_{ij}u^i +
\sum_{i=k+1}^{k+d}\sum_{j=0}^{d}G_{ij}b^ju^i + u^{k+d+1} +
\sum_{i=1}^{d}\sum_{j=i}^{d} G_{k+d+1+i,j}b^ju^{k+d+1+i}\ ,
\end{align*}
where $F_{ij},G_{ij}\in \C$.

\end{lem}

\begin{proof}
Write
\vvn-.8>
\be
U(u)\ =\ bu^{2k+d-1}\,+\,\sum_{j=0}^d b^jU_j(u)
\ee
with $U_j(u)\in \C[u]$ and
$\deg\,U_j \leq 2k+d-2$ for all $j$.

Let $y(u) = y_0(u) + by_1(u) + \dots + b^dy_d(u)$ be a solution with
$y_i(u) \in \C[u]$. Assume that $y_0(u) \neq 0$. Then $y_0(u)$ is
of degree $k$ or $k+d+1$ and $y_0(u)$ satisfies the equation
$\D_0y_0(u)=0$, where
\vvn-.4>
\be
\D_0 \ = \ \partial^2\ -\ \frac{W'}{W}\,\partial\ +\ \frac{U_0}{W}\ .
\ee

Assume that $y_0$ is of degree $k$ and monic. The polynomial $y_1(u)$
is a solution of the inhomogeneous differential equation
\beq
\D_0 y_1(u)\ +\ \frac {u^{2k+d-1} + U_1(u)}{W(u)}\,y_0(u)\ = \ 0\ .
\eeq
The term $\D_0 y_1(u)$ is of order $u^{k-1}$ as $u\to\infty$. The polynomial
$y_1(u)$ is defined up to addition of a solution of the homogeneous equation.
Therefore, $y_1(u)$ does exist and unique if it has the form
\vvn-.5>
\beq
y_1(u)\ = \ \frac {-1}{\chi(k+1)}\,u^{k+1}\ +\ \sum_{i=0}^{k-1}\,y_{i1}\,u^i\
\eeq
with $y_{i1}\in\C$.
Continuing this reasoning, we can show that a solution
$y(u) = y_0(u) + by_1(u) + \dots + b^dy_d(u)$ with $y_i(u)\in\C[u]$
does exist and unique if $y_0(u)$ is a monic polynomial of degree $k$ and
for $j=1,\dots,d$, the polynomial $y_j(u)$
has the form
\vvn.2>
\be
y_j(u)\ = \ \frac {(-1)^j}{\prod_{m=1}^j\chi(k+m)}\,u^{k+j}
\ +\ \sum_{i=0}^{k-1}\,y_{ij}\,u^i\ + \sum_{i=k+1}^{k+j-1}\,y_{ij}\,u^i\
\vv.3>
\ee
with $y_{ij}\in\C$. We take this $y(u)$ to be $F(u)$ in the lemma.
Similarly, we can construct the polynomial $G(u)$ in the lemma.
\end{proof}

\subsection{Special homomorphisms $\Ol\to\A_d$}
\label{Special homomorphisms}

Let $\{f\}(u),\{g\}(u)$ be the polynomials defined in \Ref{f g}.
Let $\D_\Ol$ be the universal differential operator defined in \Ref{diff oper}.

Let $W(u)\in \C[u]$ be a monic polynomial of degree $2k+d$. Let
$U(u)\in\A_d[u]$ be a polynomial
of the form described in
\Ref{V}. Let $\D$ and $\chi(\al)$ be defined by \Ref{diff ex} and
\Ref{chi}, respectively. Assume that all solutions to the differential
equation $\D y(u)=0$ are polynomials and $\chi(\al) =
(\al-k)(\al-k-d-1)$.
Consider the two polynomials $F(u),G(u)$, described in Lemma \ref{lem on solutions}.
Write them in the form:
\begin{align*}
F(u) & \;{}=\;\sum_{i=0}^{k-1}F_iu^i + u^k +
\sum_{i=1}^{d} \tilde F_{k+i}b^iu^{k+i}\ ,
\\[2pt]
G(u) &\;{}=\; \sum_{i=0}^{k-1} G_iu^i +
\sum_{i=k+1}^{k+d}G_iu^i + u^{k+d+1} + \sum_{i=1}^{d}
\tilde G_{k+d+1+i}b^iu^{k+d+1+i}\ ,
\end{align*}
where
\vvn-.3>
\be
F_i\,=\,\sum_{j=0}^{d} F_{ij}b^j,\ G_i=\sum_{j=0}^{d} G_{ij}b^j,
\ \tilde F_{k+i}=\sum_{j=i}^{d} F_{k+i,j}b^{j-i},\
\tilde G_{k+d+1+i}=\sum_{j=i}^{d} G_{k+d+1+i,j}b^{j-i}.
\ee

\begin{lem}
\label{basic lem}
A map
\begin{gather}
\{f_i\}\ \mapsto \ F_i\ , \quad \{g_i\}\ \mapsto\ G_i\ , \quad
\{\tilde f_{k+i}b^i\}\ \mapsto \ \tilde F_{k+i}b^i\ ,
\\[4pt]
\{\tilde g_{k+d+1+i}b^i\}\ \mapsto \ \tilde G_{k+d+1+i}b^i\ ,
\quad \{b\}\ \mapsto\ b \
\end{gather}
defines an algebra homomorphism $\eta : \Ol\to\A_d$. Under this homomorphism,
\vvn.2>
\be
\eta(\{f\}(u))\ =\ F(u)\ ,\qquad \eta(\{g\}(u))\ =\ G(u)\ ,
\qquad \eta(\D_\Ol)\ =\ \D\ .
\vv.2>
\ee
\end{lem}

Here $\eta(\{f\}(u))$ is the polynomial in $u$ obtained from
$\{f\}(u)$ by replacing the coefficients with their images in
$\A_d$. Similarly, $\eta(\{g\}(u))$ and $ \eta(\D_\Ol)$ are defined.

\begin{proof}
It is enough to prove that $\eta(\{f\}(u))=F(u),\, \eta(\{g\}(u))=G(u)$
and this follows from the definition of $\Ol$.
\end{proof}

Lemma \ref{basic lem}
assigns a homomorphism $\eta:\Ol\to \A_d$ to every differential operator
$\D$ satisfying the assumptions of Lemma \ref{lem on solutions}.

The homomorphism $\eta$ of Lemma \ref{basic lem} is such that
\vvn.2>
\beq
\label{special}
\eta (\Wr(\{f\}(u),\{g\}(u)))\ \in\ \C[u]\ .
\vv.2>
\eeq
We call an arbitrary homomorphism $\eta :\Ol\to\A_d$ {\it
a special homomorphism} if $\eta : \{b\}\mapsto b$ and
$\eta$ has property \Ref{special}.

\subsection{Special homomorphisms and Bethe eigenleaves}
\label{Special hom and leaves}
Under notations of Section \ref{Special homomorphisms},
define $n$ by the formula
$n=2k+d$. Then $d=n-2k$. Define $\bs\la = (k+d,k) = (n-k,k)$.

For $\bs a\in \C^n$, consider the $\B$-module $\Ma$ and its submodule $\Mal$,
see definitions in Section \ref{sec Deformed Ma}.
Assume that $\Mal$ has a Bethe eigenleaf $\Malp$.
Consider the universal differential operator $\D_\Malp$ of the Bethe
eigenleaf $\Malp$.
By Lemmas \ref{lem on leaf} and \ref{lem on solutions on leaf},
the differential operator $\D_\Malp$ satisfies the assumptions of
Lemma \ref{lem on solutions}, if we identify the operator
$B_{21}:\Malp \to\Malp$ in Lemmas
\ref{lem on leaf} and \ref{lem on solutions on leaf} with the element
$b\in \A_d$ in Lemma \ref{lem on solutions}.

By Lemma \ref{basic lem},
the differential operator $\D_\Malp$ determines
a special homomorphism $\eta :\Ol\to\A_d$, which will be called {\it the
special homomorphism associated with a Bethe eigenleaf}.
We have $\eta(\D_\Ol)=\D_\Malp$ by Lemma \ref{basic lem}.

\subsection{Wronski homomorphisms} Set again $n= 2k+d$.
The Wronskian $\Wr(\{f\}(u),\{g\}(u))\in\Ol[u]$ has the form
\beq
\label{Wronskian}
\Wr(\{f\}(u),\{g\}(u))\ = \ \sum_{j=0}^{n}(-1)^jW_ju^{n-j}\ ,
\eeq
with $W_j\in \Ol$ for all $j$ and $W_0 = d+1+ w_0$,
where $w_0$ is an element of the ideal $\langle\{b\}\rangle\subset \Ol$.
Thus, the coefficient $W_0$ is invertible in $\Ol$.

Let $\sigma_s$, $s=1,\dots, n$, be indeterminates. Define a grading on
$\C[\sigma_1,\dots,\sigma_n]$ by setting $\deg \sigma_s=s$ for all $s$.
The algebra homomorphism,
\vvn.2>
\be
\pi_{\bs\la}\ :\ \C[\sigma_1,\dots,\sigma_n]\ \to\ \Ol\ , \qquad
\sigma_s\ \mapsto\ \frac{W_s}{W_0}\ , \qquad s=1,\dots, n\ ,
\vv.2>
\ee
will be called the {\it Wronski homomorphism} for $\Ol$. The composition
\vvn.4>
\be
\pi_{\bs\la}^0\ =\ p_{\bs\la}^0\pi_{\bs\la}
\ :\ \C[\sigma_1,\dots,\sigma_n]\ \to\ \Olo\
\vv.4>
\ee
will be called the {\it Wronski homomorphism} for $\Olo$.
Both Wronski homomorphism $\pi_{\bs\la}^0$ are graded.

\medskip
\noindent
{\bf Remark.}\ The map
$\pi^0_{\bs\la} : \C[\sigma_1,\dots,\sigma_n] \to \Olo$ is the standard Wronski map,
see for example \cite{EG}.

\subsection{Fibers of Wronski map}
\label{comalg}
Let $A$ be a commutative $\C$-algebra. The algebra $A$ considered as an $A$-module
is called the {\it regular representation \/} of $A$. The dual space $A^*$ is
naturally an $A$-module, which is called the {\it coregular representation\/}.

A bilinear form $(\,{,}\,):A\otimes A\to\C$
is called {\it invariant\/} if $(ab,c)=(a,bc)$ for all $a,b,c\in A$.
A finite-dimensional commutative algebra $A$ with an invariant
nondegenerate symmetric bilinear form ${(\,{,}\,):A\otimes A\to\C}$ is called
a {\it Frobenius algebra\/}.

\medskip
For $\bs a\in\C^n$, let ${I^\O_\lba}$ be the ideal
in $\O_{\bs\la}\>$ generated by the elements $\pi(\sigma_s)-a_s$,
$s=1,\dots, n$.
Let
\vvn-.1>
\beq
\label{Olaa}
\O_\lba\>=\,\O_{\bs\la}/I^\O_\lba
\vv.3>
\eeq
be the quotient algebra. The algebra $\O_\lba$
is a scheme-theoretic fiber of the Wronski homomorphism.

\begin{lem}\label{local Wr}
If the algebra $\O_\lba$ is finite-dimensional, then it is a
Frobenius algebra.
\end{lem}
\begin{proof}

We have a natural isomorphism
\be
\Ol\ \simeq \ \A_d\otimes \C[\{f,g\}] \ =\
\C[f_i,g_i,b]/\langle b^{d+1}\rangle\ .
\ee
The ideal $I^\O_\lba\subset \Ol$ is generated by $n$ elements
$\pi(\sigma_s)-a_s$, $s=1,\dots, n$. Hence,
$\O_\lba$ is the quotient of the polynomial algebra
$\C[f_i,g_i,b]$ with $n+1$ generators by an ideal with
$n+1$ generators. Any such a
finite-dimensional quotient is a Frobenius algebra, see for instance,
Lemma 3.9 in \cite{MTV3}.
\end{proof}

\section{Isomorphisms}
\label{sec ISOMORR}
\subsection{Isomorphism $\tau_{\bs\la} : \Ol\to\B_{\bs\la}$}
\label{Isomorphism tau}
Let $\Vl$ be a deformed isotypical component of $\V^S$, see
Section \ref{Deformed isotypical components}.
Let $\Bl$ be
the image of $\B$ in ${\rm End} (\Vl)$.
Denote $\hat B_{ij}\in\Bl$ the image of the standard generators $B_{ij}\in\B$.

Consider a map
\vvn-.1>
\be
\label{def tau}
\tau_{\bs\la}\ :\ \O_{\bs\la}\ \to\ \B_{\bs\la}\,,
\qquad
F_{ij}\ \mapsto\ \hat B_{ij}\,,
\vv.2>
\ee
where the generators $F_{ij}$ of the algebra
$\O_{\bs\la}$ are defined in \Ref{F-ii}.
In particular,
\beq
\label{F21=e21}
\tau_{\bs\la}\ : \ \F_{21}=\{b\}\ \mapsto \hat B_{21} = e_{21}|_\Vl\ .
\eeq

\begin{thm}
\label{first}
The map $\tau_{\bs\la}$ is a well-defined isomorphism of graded algebras.
\end{thm}
\begin{proof}
Let $R(F_{ij})$ be a polynomial in generators $F_{ij}\in \Ol$ with
complex coefficients. Assume that $R(F_{ij})$ is equal to zero in
$\O_{\bs\la}$. We will prove that the corresponding polynomial $R(\hat
B_{ij})$ is equal to zero in $\B_{\bs\la}$. This will prove that
$\tau_{\bs\la}$ is a well-defined algebra homomorphism.

Consider the vector bundle over $\C^n$ with fiber $\Mal$ over a point
$\bs a$.
The polynomial
$R(\hat B_{ij})$
defines a section of the associated bundle with
fiber ${\rm End}(\Mal)$.
If $R(\hat B_{ij})$ is not equal to zero identically, then there exist
a fiber $\Mal$ and a Bethe eigenleaf $\Malp
\subset \Mal$, such that $R(\hat B_{ij}|_\Malp)\in \End(\Malp)$ is not equal to zero.
Let
\beq
\label{dd}
\D_{\Malp}\ =\ \der^2 - \bar B_1(u)\partial + \bar B_2(u)\ ,
\eeq
be the universal differential operator of the Bethe eigenleaf
$\Malp$, see \Ref{D-Malp}. Write
\beq
\label{dk}
\bar B_1(u)\ = \ \sum_{j=1}^\infty \,{\bar B_{1j} u^{-j}}\ , \qquad
\bar B_2(u)\ = \ \sum_{j=1}^\infty \, {\bar B_{2j} u^{-j}}\ .
\eeq
Then $\bar B_{ij}=\hat B_{ij}|_\Malp$ for all $(i,j)$.
Consider the special homomorphism $\eta : \Ol\to\A_d$ associated
with the Bethe eigenleaf
$\Malp$, see Sections \ref{Special homomorphisms} and
\ref{Special hom and leaves}.
By Lemma \ref{basic lem}, $\eta(\D_\Ol)=D^\Malp$. This equality
contradicts to the fact that
$R(F_{ij})$ is zero in $\Ol$ and
$R(\bar B_{ij})$ is nonzero in ${\rm End}(\Malp)$.
Thus, $R(\hat B_{ij})$ is zero in $\Bl$.

By Lemmas \ref{lem Bethe degree} and \ref{lem Ol degree},
the elements $F_{ij}$ and $\hat B_{ij}$
are of the same degree. Therefore, the homomorphism
$\tau_{\bs\la}$ is graded.

Since the elements $\hat B_{ij}$ generate the algebra $\B_{\bs\la}\,$,
the map $\tau_{\bs\la}$ is surjective.

Let $R(F_{ij})$ be a polynomial in generators $F_{ij}\in \Ol$ with complex
coefficients. Assume that $R(F_{ij})$ is a nonzero element of $\Ol$.
We will prove that the corresponding polynomial
$R(\hat B_{ij})$
is not equal to zero in $\B_{\bs\la}$. This will prove that
the homomorphism $\tau_{\bs\la}$ is injective.

Since $\Ol\simeq \C[\{f,g\}]\otimes \A_d$.
Any nonzero element $R(F_{ij})\in\Ol$ can be written in the form
\bea
R(F_{ij}) \ = \ \sum_{j=j^0}^d\, R_j(\{f_i\},\{g_i\})\,\{b\}^j\ ,
\eea
where $R_j(\{f_i\},\{g_i\}) \in \C[\{f_i\},\{g_i\}]$ and
$R_{j^0}(\{f_i\},\{g_i\})$
is a nonzero polynomial.

For generic numbers $F_i^0, G_i^0\in\C$, we have
$R_{j^0}(F^0_i,G^0_i)\neq 0$. Consider two polynomials
$F_0(u)=u^k+\sum_i F^0_iu^i$ and $G_0(u)=u^{k+d+1}+\sum_i
G^0_iu^i$. By Lemma \ref{lem on pairs of polynomials and leaves},
there exists a Bethe eigenleaf such that $\D_{\Malp,0} =
\D_{F_0,G_0}$. Let $\bar B_{ij}$ be the coefficients of $\D_\Malp$,
see \Ref{dd} and \Ref{dk}. Then $R(\bar B_{ij})\neq 0$. Hence,
$R(\hat B_{ij})$ is not equal to zero in $\B_{\bs\la}$.
\end{proof}

\subsection{Algebras $\Ol$ and $\Bl$ as $\C[\sigma_1,\dots,\sigma_n]$-modules}
\label{sec sym-modules}

The algebra $\C[z_1,\dots,z_n]^S$ $=$
\linebreak
$\C[\sigma_1,\dots,\sigma_n]$
is embedded into the algebra $\B_{\bs\la}$
as the subalgebra of operators of multiplication by symmetric polynomials,
see Lemma~\ref{Uz}.
This embedding makes $\B_{\bs\la}$ a
$\C[\sigma_1,\dots,\sigma_n]$-module.

The Wronski homomorphism $\pi_{\bs\la} : \C[\sigma_1,\dots,\sigma_n]
\to \Ol$ makes $\Ol$ a $\C[\sigma_1,\dots,\sigma_n]$-module.

\begin{lem}\label{symm OK}
The map $\,\tau_{\bs\la}:\O_{\bs\la}\to\B_{\bs\la}$ is
an isomorphism of $\,\C[\sigma_1,\dots,\sigma_n]\<$-modules, that is,
for any $s=1,\dots,n$, $\tau_{\bs\la}(\pi_{\bs\la}(\sigma_s))$
is the operator of multiplication by $\sigma_s$.
\end{lem}
\begin{proof}
The proof follows from the two formulae:
\beq
\label{F1}
B_1(u)\ = \ e_{11}(u) + e_{22}(u)\ ,
\qquad
F_1(u)\,=\,-\,\frac{\Wr'(\{f\}(u),\{g\}(u))}
{\Wr(\{f\}(u),\{g\}(u))}\ .
\eeq
\end{proof}

\begin{cor}
\label{cor emb}
The Wronski homomorphism $\pi_{\bs\la} : \C[\sigma_1,\dots,\sigma_n]
\to \Ol$ is an embedding.
\end{cor}

Consider the projection
$p^\O_{\bs\la} : \Ol \to \Olo$ defined in \Ref{pOl}. The composition
\vvn.3>
\be
\pi^0_{\bs\la} = p^\O_{\bs\la} \pi_{\bs\la}\ :\
\C[\sigma_1,\dots,\sigma_n]\ \to\ \Olo
\vv.3>
\ee
is the standard Wronski map.
Its degree $d^0_{\bs\la}$ is given by the Schubert calculus.
In particular, we have
\beq
\label{d0}
(d+1)\,d^0_{\bs\la}\ = \ \dim\, (V^{\otimes n})_{\bs\la}\ ,
\eeq
where $(V^{\otimes n})_{\bs\la}\subset V^{\otimes n}$ is the
$\glt$-isotypical component corresponding to the irreducible
polynomial $\glt$-representation with highest weight $\bs\la
=(n-k,k)$ and $d=n-2k$.

\begin{prop}
\label{prop on f.d.}
For $\bs a\in\C^n$, let ${I^\O_\lba}$ be the ideal in $\O_{\bs\la}\>$
generated by the elements $\pi(\sigma_s)-a_s$, $s=1,\dots, n$.
Let $\O_\lba\>=\,\O_{\bs\la}/I^\O_\lba$ be the quotient algebra. Then
\vvn.2>
\be
\dim\,
\O_\lba\ = \
\dim\, (V^{\otimes n})_{\bs\la}\ .
\ee
\end{prop}
\begin{proof}
The proposition follows from Lemma~\ref{aux lem}
\end{proof}

Let $H_s(x_1,\dots,x_m,b)$, $s=1,\dots,m$, be $m$ polynomials in
$\C[x_1,\dots,x_m,b]$ such that
\be
H_s(x_1,\dots,x_m,b)\ = \ \sum_{j=0}^d H_{sj}(x_1,\dots,x_m)\,b^j\ .
\ee
Let $I\subset \C[x_1,\dots,x_m,b]$ be the ideal generated by $m+1$ polynomials:
$b^{d+1}$ and \>$H_s(x_1,\dots,\alb x_m,\alb b)$, \,$s=1,\dots,m$.
Let $I_0\subset \C[x_1,\dots,x_m]$ be the ideal generated by the polynomials
$H_{s0}(x_1,\dots,x_m)$, \,$s=1,\dots,m$.

\begin{lem}
\label{aux lem}
Assume that $\C[x_1,\dots,x_m]/I_0$ is finite-dimensional. Then
\vvn.3>
\be
\rightline{\hfill\qedph$\dim \C[x_1,\dots,x_m,b]/I\ = \
(d+1)\, (\dim \C[x_1,\dots,x_m]/I_0)\ .$\qed}
\vv.3>
\ee
\end{lem}

\subsection{Isomorphism $\mu_{\bs\la}:\Ol \to \V^S_{\bs\la}$}
\label{sec iso mu}

By Lemma \ref{lem on isotyp decomp}, the space $\V^S_{\bs\la}$ is
a graded free
\linebreak
$\C[\sigma_1,\dots,\sigma_n]$-module.
It has a unique (up to proportionality) vector of degree ${2k-n}$.
Fix such a vector $v_{\bs\la}\in\V^S_{\bs\la}$ and
consider a linear map
\vvn.3>
\be
\mu_{\bs\la}\ :\
\O_{\bs\la}\to\V^S_{\bs\la}\,,\qquad F \mapsto\tau_{\bs\la}(F)\,v_{\bs\la}\,.
\vv.3>
\ee

\begin{thm}
\label{first1}
The map\/ \>$\mu_{\bs\la}:\O_{\bs\la}\to\V^S_{\bs\la}$ is an isomorphism
of graded vector spaces. The maps\/
$\tau_{\bs\la}$ and\/ $\mu_{\bs\la}$ intertwine the action of multiplication
operators on $\O_{\bs\la}$ and the action of the Bethe algebra $\B_{\bs\la}$
on $\V^S_{\bs\la}$, that is, for any $F,G\in\O_{\bs\la}$, we have
\vvn.4>
\beq
\label{mutau}
\mu_{\bs\la}(FG)\,=\,\tau_{\bs\la}(F)\,\mu_{\bs\la}(G)\,.
\vv.3>
\eeq
In other words, the maps\/ $\tau_{\bs\la}$ and\/ $\mu_{\bs\la}$ give
an isomorphism of the regular representation of\/ $\O_{\bs\la}$ and
the\/ $\B_{\bs\la}$-module $\V^S_{\bs\la}$.
\end{thm}
\begin{proof}
For any nonzero $H\in \C[\sigma_1,\dots,\sigma_n]$, the vector
$(B_{21})^dHv_{\bs\la}$ is a nonzero vector.
Thus, the kernel of $\mu_{\bs\la}$ is an ideal
$I$ in $\B_{\bs\la}$, which does not contain elements of the form
$(B_{21})^dH$. Hence,
$\tau_{\bs\la}^{-1}(I)$ is an ideal in $\Ol$, which
does not contain elements of the form
$\{b\}^d\tilde H$, where $\tilde H\in \pi(\C[\sigma_1,\dots,\sigma_n])$.
It is easy to see that any ideal in $\Ol$, which does not contain elements of
the form $\{b\}^d\tilde H$, is the zero ideal.
This reasoning proves that $\mu_{\bs\la}$ is injective.

The map $\mu_{\bs\la}$ is a graded linear map.
We have the equality of characters,
$\ch_{\V^{S}_{\bs\la}}(q) = \ch_\Ol(q)$, due to
formulae \Ref{char VS} and \Ref{ch Ol}.
Hence, the map $\mu_{\bs\la}$ is surjective.
Formula~\Ref{mutau} follows from Theorem~\ref{first}.
\end{proof}

\subsection{Isomorphism of algebras $\O_\lba$ and $\B_\lba$}
Let $\bs a=(a_1,\dots,a_n) \in \C^n$.
Consider the $\B$-module $\Mal$. Denote $\Bal$ the image of $\B$ in
$\End (\Mal)$.

Let $I^\B_\lba\subset\B_{\bs\la}$ be the ideal generated by
the elements $\si_s(\bs z)-a_s$, $s=1,\dots,n$. Consider the subspace
$\,I^\V_\lba=I^\B_\lba\>\V^S_{\bs\la}$.

\begin{lem}
\label{identify}
We have
\vvn.2>
\be
\tau_{\bs\la}({I^\O_\lba})={I^\B_\lba}\,,\qquad
\mu_{\bs\la}({I^\O_\lba})={I^\V_\lba}\,,\qquad
\B_\lba=\B_{\bs\la}/{I^\B_\lba}\,,\qquad
\M_\lba=\V^S_{\bs\la}/{I^\V_\lba}\,.
\vv.4>
\ee
\end{lem}
\begin{proof}
The lemma follows from Theorems~\ref{first}, \ref{first1} and
Lemmas~\ref{symm OK}, \ref{factor=weyl}.
\end{proof}

By Lemma~\ref{identify},
the maps $\tau_{\bs\la}$ and $\mu_{\bs\la}$ induce the maps
\beq
\label{taumu}
\tau_\lba:\O_\lba\to\B_\lba\,,\qquad
\mu_\lba:\O_\lba\to\M_\lba\,.
\eeq

\begin{thm}
\label{second}
The map $\tau_\lba$ is an isomorphism of algebras. The map $\mu_\lba$ is
an isomorphism of vector spaces. The maps\/ $\tau_\lba$ and\/ $\mu_\lba$
intertwine the action of multiplication operators on $\O_\lba$ and the action
of the Bethe algebra $\B_\lba$ on $\M_\lba$, that is, for any $F,G\in\O_\lba$,
we have
\vvn-.2>
\be
\mu_\lba(FG)\,=\,\tau_\lba(F)\,\mu_\lba(G)\,.
\vv.3>
\ee
In other words, the maps\/ $\tau_\lba$ and\/ $\mu_\lba$ give an isomorphism of
the regular representation of\/ $\O_\lba$ and the\/ $\B_\lba$-module $\M_\lba$.
\end{thm}
\begin{proof}
The theorem follows from Theorems~\ref{first}, \ref{first1}
and Lemma~\ref{identify}.
\end{proof}

\begin{rem}
By Lemma~\ref{local Wr}, the algebra $\O_\lba$ is Frobenius.
Therefore, its regular and coregular representations are isomorphic.
\end{rem}

\section{Comparison of actions of $\B$ and $\B^0$ on $\V^S$}
\label{sec comparison}

\subsection{Isomorphism $\nu_{\bs\la} : \A_d\otimes \B^0_{\bs\la} \to
\B_{\bs\la}$}
\label{isom nu}

\begin{lem}
\label{lem on ker p}

Consider the principal ideal $\langle \hat B_{21}\rangle \subset
\B_{\bs\la}$ and the graded algebra
epimorphism $p^\B_{\bs\la}:\B_{\bs\la}\to\B_{\bs\la}^0$, defined in \Ref{p-la-b}.
Then
$\langle \hat B_{21}\rangle = {\rm ker}\,p^\B_{\bs\la}$.

\end{lem}

\begin{proof}
Clearly,
we have
$\langle \hat B_{21}\rangle \subset {\rm ker}\,p^\B_{\bs\la}$.
Consider the commutative diagram of algebra homomorphisms,
\vvn-.4>
\beq
\label{2nd diagram}
\begin{CD}
\A_d @> i^\O_{\bs\la}
>> \Ol
@>p_{\bs\la}^\O >> \Olo
\\
{\rm id}@VVV \tau_{\bs\la} @VVV
\\
\A_d @>i^\B_{\bs\la} >> \B_{\bs\la} @> p^\B_{\bs\la}>> \B^0_{\bs\la}
\end{CD} \phantom{aa} .
\eeq
We have ${\rm ker}\,p^\O_{\bs\la} = \langle \{b\}\rangle$.
The graded characters of $\Olo$ and $\B^0_{\bs\la}$ are equal due to
\Ref{char B^0_{la}}, \Ref{ch Ol}.
Hence
$\langle \hat B_{21}\rangle = {\rm ker}\,p^\B_{\bs\la}$.
\end{proof}

\begin{cor}
\label{lem on isom tau0}
The isomorphism $\tau_{\bs\la}$ induces an isomorphism
\vvn.3>
\be
\tau_{\bs\la}^0\ :\ \Olo\ \to\ \B^0_{\bs\la}\ .
\ee
\end{cor}

\medskip
\noindent
{\bf Remark.} The isomorphism $\tau_{\bs\la}^0 : \Olo \to \B^0_{\bs\la}$
is the isomorphism denoted $\tau_{\bs\la}$ in Theorem 5.3 of \cite{MTV3}.

Denote
\vvn-.7>
\be
r^\O_{\bs\la} \ :\ \A_d \otimes \Olo\ \to\ \Olo
\vv.3>
\ee
the algebra epimorphism such that
$ b\otimes x \mapsto 0, \ 1\otimes x \mapsto x$ for any $x\in \Olo$.
Denote
\vvn.3>
\be
r^\B_{\bs\la}\ :\ \A_d \otimes \B^0_{\bs\la} \ \to\ \B^0_{\bs\la}
\vv.3>
\ee
the algebra epimorphism such that
$ b\otimes x \mapsto 0, \ 1\otimes x \mapsto x$ for any $x\in \B^0_{\bs\la}$.

\begin{thm}
\label{thm on 3d diagram}
The following diagram
is commutative,
\beq
\label{3d diagram}
\begin{CD}
\A_d @> {\rm id}\otimes 1
>>
\A_d\otimes \B^0_{\bs\la} @>r^\B_{\bs\la} >> \B^0_{\bs\la}
\\
{\rm id}@VVV \nu_{\bs\la}
@VVV {\rm id}@VVV
\\
\A_d @>i^\B_{\bs\la} >>
\B_{\bs\la} @> p^\B_{\bs\la}>> \B^0_{\bs\la}
\end{CD} ,
\eeq
where $\nu_{\bs\la}$ is the isomorphism defined by the formula
$\nu_{\bs\la}\, = \,
\tau_{\bs\la}\,q_{\bs\la}\,({\rm id}\otimes (\tau_{\bs\la}^0)^{-1})$.
\end{thm}

\begin{proof}
The theorem follows from the commutativity of the following diagram:
\be
\label{4th diagram}
\begin{CD}
\A_d
@> {\rm id}\otimes 1 >>
\A_d\otimes \B^0_{\bs\la}
@>r^\B_{\bs\la} >>
\B^0_{\bs\la}
\\
{\rm id}@VVV
{\rm id} \otimes (\tau_{\bs\la}^0)^{-1}
@VVV (\tau_{\bs\la}^0)^{-1} @VVV
\\
\A_d @> {\rm id}\otimes 1 >>
\A_d\otimes \Olo
@> r^\O_{\bs\la}>> \Olo
\\
{\rm id}@VVV
q_{\bs\la}
@VVV {\rm id}
@VVV
\\
\A_d @> i^\O_{\bs\la}
>>
\Ol
@> p^\O_{\bs\la}>> \Olo
\\
{\rm id}@VVV
\tau_{\bs\la}
@VVV
\tau_{\bs\la}^0
@VVV
\\
\A_d @> i^\B_{\bs\la}
>>
\Bl
@> p^\B_{\bs\la}>> \B^0_{\bs\la}
\end{CD}\ {} \ .
\ee
\end{proof}

\subsection{$\A_d\otimes \B^0_{\bs\la}$-module $\V^{S,0}_{\bs\la}$}
\label{A-B-module V{S,O}}

By Lemma \ref{lem on char of VS}, the space $\sing\,\V^{S,0}_{\bs\la}$ is
a graded free
$\C[\sigma_1,\dots,\sigma_n]$-module.
It has a unique (up to proportionality) vector of degree ${2k-n}$.
Fix such a vector $v^0_{\bs\la}\in\,\sing\,\V^{S,0}_{\bs\la}$ and
consider a linear map
\bea
\mu^0_{\bs\la}\ :\
\O^0_{\bs\la}\ \to\
\sing \V^{S,0}_{\bs\la}\,,\qquad F \mapsto
\tau^0_{\bs\la}(F)\,v^0_{\bs\la}\,.
\eea

\begin{thm}[Theorem 5.6 of \cite{MTV3}]
\label{first2}
The map $\mu^0_{\bs\la}$
is an isomorphism
of graded vector spaces. The maps\/
$\tau^0_{\bs\la}$ and\/ $\mu^0_{\bs\la}$ intertwine the action of multiplication
operators on $\O^0_{\bs\la}$ and the action of the Bethe algebra $\B^0_{\bs\la}$
on $\sing \V^S_{\bs\la}$, that is, for any $F,G\in\O^0_{\bs\la}$, we have
\vvn.3>
\be
\mu^0_{\bs\la}(FG)\,=\,\tau^0_{\bs\la}(F)\,\mu^0_{\bs\la}(G)\,.
\vv.3>
\ee
In other words, the maps\/ $\tau^0_{\bs\la}$ and\/ $\mu^0_{\bs\la}$ give
an isomorphism of the regular representation of\/ $\O^0_{\bs\la}$ and
the\/ $\B^0_{\bs\la}$-module $\sing \V^{S,0}_{\bs\la}$.
\end{thm}

Consider the linear map
\bea
\bar \mu^0_{\bs\la}\ :\
\A_d\otimes \O^0_{\bs\la}\ \to\ \V^{S,0}_{\bs\la}\ ,
\qquad
b^j\otimes F\ \mapsto\ (e_{21})^j\tau^0_{\bs\la}(F)\,v^0_{\bs\la}\,,
\eea
and the algebra isomorphism
\bea
{\rm id}\otimes \tau^0_{\bs\la}\ : \
\A_d\otimes \O^0_{\bs\la}\ \to \
\A_d\otimes \B^0_{\bs\la}\ .
\eea

\begin{cor}
\label{cor on reg repn 0}
The map\ $\bar \mu^0_{\bs\la}$ is an isomorphism of graded vector
spaces. The maps\/ ${\rm id}\otimes \tau^0_{\bs\la}$ and\/ $\bar
\mu^0_{\bs\la}$ intertwine the action of multiplication operators on
$\A_d\otimes \O^0_{\bs\la}$ and the action of the algebra
$\A_d\otimes\B^0_{\bs\la}$ on $\V^S_{\bs\la}$, that is, for any
$F,G\in\O^0_{\bs\la}$ and $i,j\geq 0$, we have
\vvn.3>
\be
\bar \mu^0_{\bs\la}(b^{i+j}\otimes FG)
\ =
({\rm id}\otimes
\tau^0_{\bs\la})(b^i\otimes F)\,\bar \mu^0_{\bs\la}(b^j\otimes G)\,.
\vv.3>
\ee
In other words, the maps\/ ${\rm id}\otimes \tau^0_{\bs\la}$ and\/
$\bar \mu^0_{\bs\la}$ give
an isomorphism of
the regular representation of\/ $\A_d\otimes \O^0_{\bs\la}$ and
the\/ $\A_d\otimes \B^0_{\bs\la}$-module $\V^{S,0}_{\bs\la}$
defined in Section \ref{Algebra Adotimes Bbsla}.
\end{cor}

\subsection{Comparison of
$\A_d\otimes \B^0_{\bs\la}$-module $\V^{S,0}_{\bs\la}$
and $\B_{\bs\la}$-module $\V^S_{\bs\la}$}
\label{last subsec}

Define a linear map
\bea
\eta_{\bs\la}\ :\ \V^{S,0}_{\bs\la} \ \to\ \V^{S}_{\bs\la}
\eea
by the formula
\bea
(e_{21})^j B\,v^0_{\bs\la}\ \mapsto\
\nu_{\bs\la}(b^j\otimes B) v_{\bs\la}\
\eea
for any $j\geq 0$ and $B\in \B^0_{\bs\la}$.

\begin{thm}
\label{last thm}
The map $\eta_{\bs\la}$
is an isomorphism
of graded vector spaces. The maps\/
$\nu_{\bs\la} : \A_d\otimes \B^0_{\bs\la}\to \B_{\bs\la}$
and\/ $\eta_{\bs\la}$ intertwine the action of
$\A_d\otimes\B^0_{\bs\la}$ on $\V^{S,0}$
and the action of $\B_{\bs\la}$
on $\V^S_{\bs\la}$.
In other words, the maps\/
$\nu_{\bs\la}$
and\/ $\eta_{\bs\la}$
give
an isomorphism of the
$\A_d\otimes \B^0_{\bs\la}$-module $\V^{S,0}_{\bs\la}$ and
$\B_{\bs\la}$-module $\V^S_{\bs\la}$.
\end{thm}

The theorem is a direct corollary of Theorems
\ref{first1}, \ref{thm on 3d diagram} and Corollary
\ref{cor on reg repn 0}.

\end{document}